\documentclass{amsart}
\usepackage{amsmath,amssymb,amscd}
\usepackage[enableskew]{youngtab}
\usepackage{graphicx}
\usepackage[usenames]{color}

\definecolor{red}{cmyk}{0,1,1,0}

\newcommand{\alt}{\tilde{\alpha}}
\newcommand{\dia}{\diamondsuit}
\newcommand{\El}{\mathbf{L}}
\newcommand{\geh}{\mathfrak{g}}
\newcommand{\gehc}{\overset{\lower3pt\hbox{${\scriptstyle \circ}$}}{\geh}}
\newcommand{\gehz}{\geh_{\overline{0}}}
\newcommand{\la}{\lambda}
\newcommand{\La}{\Lambda}
\newcommand{\Lab}{\overline{\Lambda}}
\newcommand{\Lat}{\tilde{\La}}
\newcommand{\ol}{\overline}
\newcommand{\ot}{\otimes}
\newcommand{\Pb}{\overline{P}}
\newcommand{\Pt}{\tilde{P}}
\newcommand{\qbin}[2]{\genfrac{[}{]}{0pt}{}{#1}{#2}}

\newcommand{\wt}{\mathrm{wt}\,}
\newcommand{\Z}{\mathbb{Z}}

\newcommand{\cell}{{\setlength{\unitlength}{.4mm}
\begin{picture}(4,4)\multiput(0,0)(4,0){2}{\line(0,1){4}}
\multiput(0,0)(0,4){2}{\line(1,0){4}}\end{picture}}}
\newcommand{\cells}{{\setlength{\unitlength}{.28mm}
\begin{picture}(4,4)\multiput(0,0)(4,0){2}{\line(0,1){4}}
\multiput(0,0)(0,4){2}{\line(1,0){4}}\end{picture}}}

\newcommand{\vdom}{{\Yvcentermath1\Yboxdim4pt \,\yng(1,1)\,}}
\newcommand{\hdom}{{\Yvcentermath1\Yboxdim4pt \,\yng(2)\,}}
\newcommand{\sdom}{{\Yvcentermath1\Yboxdim4pt \,\yng(1)\,}}
\newcommand{\gnode}{
\color[cmyk]{0,0,0,0.3}
\put(-0.3,0){\rule{9.8pt}{9.8pt}}
\color{black}}

\numberwithin{equation}{section}

\newtheorem{theorem}{Theorem}
\newtheorem{prop}[theorem]{Proposition}
\newtheorem{lemma}[theorem]{Lemma}
\newtheorem{cor}[theorem]{Corollary}

\theoremstyle{definition}
\newtheorem{definition}{Definition}
\newtheorem{remark}{Remark}
\newtheorem{example}[remark]{Example}

\numberwithin{theorem}{section}
\numberwithin{definition}{section}
\numberwithin{remark}{section}

\begin{document}

\title[Stable rigged configurations]
{Stable rigged configurations for quantum affine algebras
 of nonexceptional types}

\author[M.~Okado]{Masato Okado}
\address{Department of Mathematical Science,
Graduate School of Engineering Science, Osaka University,
Toyonaka, Osaka 560-8531, Japan}
\email{okado@sigmath.es.osaka-u.ac.jp}

\author[R.~Sakamoto]{Reiho Sakamoto}
\address{Department of Physics, Tokyo University of Science,
Kagurazaka, Shinjuku-ku, Tokyo 162-8601, Japan}
\email{reiho@rs.tus.ac.jp}

\thanks{\textit{Date:} May 20, 2011.}

\begin{abstract}
For an affine algebra of nonexceptional type in the large rank
we show the fermionic formula depends only on the attachment of the
node $0$ of the Dynkin diagram to the rest, and the fermionic formula
of not type $A$ can be expressed as a sum of that of type $A$ with 
Littlewood-Richardson coefficients. Combining this result with \cite{KSS} and \cite{LOS}
we settle the $X=M$ conjecture under the large rank hypothesis.
\end{abstract}

\maketitle

\section{Introduction}

Let $\geh$ be an affine Lie algebra and $I$ the index set of its Dynkin nodes. Let $\gehc$
be the classical subalgebra of $\geh$, namely, the finite-dimensional simple Lie algebra whose 
Dynkin nodes are given by $I_0:=I\setminus\{0\}$ where the node $0$ is taken as in \cite{Kac}. 
Let $U'_q(\geh)$ be the quantized enveloping algebra associated to $\geh$ without the degree 
operator. Among finite-dimensional $U'_q(\geh)$-modules there is a distinguished family called 
Kirillov-Reshetikhin (KR) modules, which have nice properties such as $T$($Q,Y$)-systems, 
fermionic character formulas, and so on. See for instance \cite{DK,Her,KNT,Nak} and references therein.
In \cite{HKOTY,HKOTT}, assuming the existence of crystal basis $B^{r,s}$ ($r\in I_0,s\in\Z_{>0}$)
of a KR module we defined the one-dimensional (1-d) sum 
\[
X_{\la,B}(q)=\sum_{b\in B}q^{D(b)}
\]
where the sum is over $I_0$-highest weight vectors in $B=B^{r_1,s_1}\ot\cdots\ot B^{r_m,s_m}$
with weight $\la$ and $D$ is a certain $\Z$-valued function on $B$ called the energy function
(see e.g. (3.9) of \cite{HKOTT}), 
and conjectured that $X$ has an explicit expression $M$ (see \eqref{fermi}) called the fermionic 
formula ($X=M$ conjecture). This conjecture is 
settled in full generality if $\geh=A_n^{(1)}\,$\cite{KSS}, when $r_j=1$ for all $j$ if $\geh$
is of nonexceptional affine types \cite{SS}, and when $s_j=1$ for all $j$ if $\geh=D_n^{(1)}$\,\cite{Sch}.
It should also be noted that recently the existence of KR crystals for nonexceptional affine types were
settled \cite{O,OS} and their combinatorial structures were clarified \cite{FOS}.

Another interesting equality related to $X$ is the $X=K$ conjecture by Shimozono and Zabrocki
\cite{SZ,Sh} that originated from the study of certain $q$-deformed operators on the ring of
symmetric functions. Suppose $\geh$ is of nonexceptional type. If the rank of $\geh$ is sufficiently
large, $X$ does not depend on $\geh$ itself, but only on the attachment of the affine Dynkin node
$0$ to the rest of the Dynkin diagram. See Table \ref{tab:diamond}. Let $X_{\la,B}^\dia(q)$ 
($\dia=\varnothing,\cell,\hdom,\vdom$) denote the 1-d sum for $\geh$ of kind $\dia$. Then the
$X=K$ conjecture, which has been settled in \cite{Sh,LS,LOS}, states that if $\dia\ne\varnothing$,
the following equality holds.
\begin{equation} \label{X=K}
X^\dia_{\la,B}(q)=q^{-\frac{|B|-|\la|}{|\dia|}}
\sum_{\mu\in\mathcal{P}^\dia_{|B|-|\la|},\eta\in\mathcal{P}^\cells_{|B|}}c_{\la\mu}^\eta
X^\varnothing_{\eta,B}(q^\frac2{|\dia|})
\end{equation}
Here $|B|=\sum_{i=1}^mr_is_i$, $\mathcal{P}_N^\dia$ is the set of partitions of $N$ whose diagrams
can be tiled from $\dia$, and $c_{\la\mu}^\eta$ is the Littlewood-Richardson coefficient. Note also 
that $\ol{X}^\dia_{\la,B}(q)$ in \cite{LOS} is related to our $X^\dia_{\la,B}(q)$ by 
$\ol{X}^\dia_{\la,B}(q)=X^\dia_{\la,B}(q^{-1})$. Let us
sketch out the proof in \cite{LOS}. Since the 1-d sum depends only on $\dia$, for each $\dia$ we choose 
$\geh=\geh^\dia$ such that $i\mapsto n-i$ for $i\in I$ induces a Dynkin diagram automorphism.
Namely, we choose $\geh^\dia=D_{n+1}^{(2)},C_n^{(1)},D_n^{(1)}$ for $\dia=\cell,\hdom,\vdom$. Then
on each KR crystal $B^{r,s}$ for $\geh^\dia$ one can show there exists an automorphism $\sigma$
satisfying $\sigma\circ\tilde{e}_i=\tilde{e}_{n-i}\circ\sigma$ for any $i\in I$, where $\tilde{e}_i$
is the Kashiwara operator. This automorphism $\sigma$ can be extended to the tensor product $B$ of 
KR crystals. Let $High$ be the map sending an element $b$ of $B$ to the $I_0$-highest weight vector of 
the $I_0$-highest component containing $b$. Under the assumption that the rank is large, the 
composition $High\circ\sigma$ has the following properties.
\begin{itemize}
\item[(i)] The image of an $I_0$-highest weight vector can be regarded as an element for type $A_n^{(1)}$.
\item[(ii)] For an $I_0$-highest weight vector $b$ for type $A_n^{(1)}$ of weight $\eta$, the number
	of $I_0$-highest weight vectors of weight $\la$ in the inverse image of $b$ is given by 
	$\sum_{\mu\in\mathcal{P}^\dia_{|B|-|\la|}}c_{\la\mu}^\eta$.
\item[(iii)] For an $I_0$-highest weight vector $b$ for $\geh^\dia$, we have $D(b)=D(\sigma(b))
	-\frac{|B|-|\la|}{|\dia|}$.
\end{itemize}
\eqref{X=K} is a direct consequence of these properties.

If we believe the $X=M$ conjecture, we have the right to expect exactly the same relation in the $M$
side under the same assumption of the rank. This is what we wish to clarify in this paper. Namely, if 
$\geh$ is one of nonexceptional affine type and the rank is sufficiently large, we show the fermionic
formula depends only on the symbol $\dia$, denoted by $M^\dia(\la,\El;q)$, and if $\dia\ne\varnothing$
we have (Theorem \ref{th:M=K})
\begin{equation} \label{M=K}
M^\dia(\la,\El;q)=q^{-\frac{|\El|-|\la|}{|\dia|}}
\sum_{\mu\in\mathcal{P}^\dia_{|\El|-|\la|},\eta\in\mathcal{P}^\cells_{|\El|}}c_{\la\mu}^\eta
M^\varnothing(\eta,\El;q^\frac2{|\dia|}).
\end{equation}
Here $\El=(L_i^{(a)})_{a\in I_0,i\in\Z_{>0}}$ is a datum such that $L_i^{(a)}$ counts the number of
$B^{a,i}$ in $B$ and $|\El|=\sum_{a\in I_0,i\in\Z_{>0}}aiL_i^{(a)}$.

The proof of \eqref{M=K} proceeds as follows. We first rewrite the fermionic formula as
\[
M^\dia(\la,\El;q)=\sum_{(\nu^\bullet,J^\bullet)\in\mathrm{RC}^\dia(\la,\El)}q^{c(\nu^\bullet,J^\bullet)}
\]
by introducing the notion of stable rigged configuration. We then construct for $\dia\ne\varnothing$ a
bijection 
\[
\Psi:\mathrm{RC}^\dia(\la,\El)\longrightarrow 
\bigsqcup_{\mu\in\mathcal{P}^\dia_{|\El|-|\la|},\eta\in\mathcal{P}^\cells_{|\El|}}
\mathrm{RC}^\varnothing(\eta,\El)\times LR_{\la\mu}^\eta,
\]
where $LR_{\la\mu}^\eta$is the set of Littlewood-Richardson skew tableaux of shape $\eta/\la$ and
weight $\mu$. Roughly speaking, the bijection $\Psi$ proceeds as follows. When the rank 
is sufficiently large, there exists $k$ such that the $a$-th configuration $\nu^{(a)}$ is the same
for $a=k,k+1,\ldots$. As opposed to the KKR algorithm that removes a box from $\nu^{(a)}$ starting
from $a=1$, we perform a similar algorithm starting from the largest $a$. If we continue this 
procedure until all boxes are removed from $\nu^{(a)}$ for sufficiently large $a$, we can regard
this as a rigged configuration of type $A$. Reflecting this sequence of procedures we can also 
define a recording tableau, that is shown to be a Littlewood-Richardson skew tableau. This map can 
be reversed at each step, and therefore defines a bijection.

Finally we show 
\[
c(\nu^\bullet,J^\bullet)=c(\nu'^\bullet,J'^\bullet)-\frac{|\El|-|\la|}{|\dia|}
\]
where $(\nu'^\bullet,J'^\bullet)$ is the first component of the image of $(\nu^\bullet,J^\bullet)$
by $\Psi$. We note that two equalities \eqref{X=K} and \eqref{M=K} together with the result of 
\cite{KSS} implies 
\[
X^\dia_{\la,B}(q)=M^\dia(\la,\El;q)
\]
for $\dia\ne\varnothing$ and therefore settle the $X=M$ when $\geh$ is of nonexceptional type and 
the rank is sufficiently large.

Let us summarize the combinatorial bijections that are
relevant to our paper as the following
schematic diagram:
\medskip
\begin{center}
\unitlength 10pt
\begin{picture}(25,8)
\put(0,7){\{type $\mathfrak{g}$ path\}}
\put(0.3,0){\{type $\mathfrak{g}$ RC\}}
\put(3,3.2){\vector(0,1){3}}
\put(3,4.2){\vector(0,-1){3}}
\put(1,3.5){(b)}
\put(15.3,0){\{type $A^{(1)}_n$ RC\} $\times$ LR}
\put(15,7){\{type $A^{(1)}_n$ path\} $\times$ LR}
\put(20,3.2){\vector(0,1){3}}
\put(20,4.2){\vector(0,-1){3}}
\put(20.5,3.5){(a)}
\put(10,0.2){\vector(1,0){5}}
\put(11,0.2){\vector(-1,0){5}}
\put(10,0.7){$\Psi$}
\end{picture}
\end{center}
\vspace{3mm}
Here ``path" stands for the highest weight elements of $\bigotimes_iB^{r_i,s_i}$
and ``RC" stands for the rigged configurations.
Our bijection $\Psi$, that exists when the rank is large, corresponds to the bottom edge.
Bijection (a), which we call type $A^{(1)}_n$ RC-bijection,
is established in full generality in the papers \cite{KKR,KR,KSS}.
Algorithms for bijection (b) are known explicitly in the following cases:
\begin{itemize}
\item $(B^{1,1})^{\otimes L}$ type paths for all
nonexceptional algebras $\mathfrak{g}$ \cite{OSS},
\item $\bigotimes B^{r_i,1}$ type paths for $\mathfrak{g}=D^{(1)}_n$
\cite{Sch},
\item $\bigotimes B^{1,s_i}$ type paths for all
nonexceptional algebras $\mathfrak{g}$ \cite{SS}.
\end{itemize}
For the cases that the bijection (b) is established,
our bijection $\Psi$ thus gives the combinatorial bijection
between the set of type $\mathfrak{g}$ paths and
the product set of the type $A^{(1)}_n$ paths and the
Littlewood-Richardson skew tableaux.
We refer to \cite{Sh} for related combinatorial problems.

We expect that the bijection (b) exists in full generality
even without the large rank hypothesis.
It will give a combinatorial proof of the $X=M$ conjecture. Furthermore, it also 
gives an essential tool
for the study of a tropical integrable system known as
the box-ball system (see e.g., \cite{FOY,HHIKTT,HKT})
which is a soliton system defined on the paths and is
supposed to give a physical background for the $X=M$ identities.
More precisely, the rigged configurations are identified with the
complete set of the action and angle variables for the type
$A^{(1)}_n$ box-ball system \cite{KOSTY}
(see \cite{KSY2} for a generalization to type $D^{(1)}_n$).
It is also interesting to note that by introducing a tropical analogue of
the tau functions in terms of the charge $c(\nu^\bullet,J^\bullet)$,
the initial value problem for the type $A^{(1)}_n$ box-ball
systems is solved in \cite{KSY1,Sak1}.
Therefore the construction of the bijection (b) in full generality will be
a very important future problem.

\section*{Acknowledgments}
\smallskip\par\noindent
M.O. thanks Ghislain Fourier, C\'edric Lecouvey, Anne Schilling and Mark Shimozono for
fruitful collaborations related to this work.
Without them this work would not have been possible. 
R.S. would like to thank Anatol N. Kirillov,
Anne Schilling and Mark Shimozono for valuable discussion related to this work.
The authors are partially supported by the Grants-in-Aid for Scientific Research 
No. 20540016 (M.O.) and No. 21740114 (R.S.) from JSPS.

\section{Stable rigged configurations}

\subsection{Affine algebras} \label{subsec:affine}
We recall necessary notations for affine Kac-Moody algebras.
We adopt the notation of \cite{HKOTT}. Let
$\geh$ be a Kac-Moody Lie algebra of nonexceptional affine type
$X^{(r)}_N$, that is, one of the types $A^{(1)}_n (n \ge 1)$,
$B^{(1)}_n (n \ge 3)$, $C^{(1)}_n (n \ge 2)$, $D^{(1)}_n (n \ge 4)$, $A^{(2)}_{2n}(n\ge1)$, 
$A^{(2)}_{2n-1}(n\ge2)$, $D^{(2)}_{n+1}(n\ge2)$. The nodes of the Dynkin diagram of $\geh$
are labeled by the set $I=\{0,1,2\dotsc,n\}$. See Table 2.1 of \cite{HKOTT}.
Let $\alpha_i,h_i,\La_i$ ($i \in I$) be the simple roots, simple
coroots, and fundamental weights of $\geh$. Let $\delta$ and $c$
denote the generator of imaginary roots and the canonical central
element, respectively. Recall that $\delta=\sum_{i\in
I}a_i\alpha_i$ and $c = \sum_{i \in I}a^\vee_i h_i$, where the Kac
labels $a_i$ are the unique set of relatively prime positive
integers giving the linear dependency of the columns of the Cartan
matrix $A$, that is, $A\,{}^t\!(a_0,\dotsc,a_n) = 0$. Explicitly,
\begin{equation}
\delta =
\begin{cases}
\alpha_0+\cdots + \alpha_n & \text{if $\geh=A^{(1)}_n$} \\
\alpha_0+\alpha_1+2\alpha_2+\cdots+2\alpha_n & \text{if $\geh=B^{(1)}_{n}$} \\
\alpha_0+2\alpha_1+\cdots+2\alpha_{n-1}+\alpha_n & \text{if $\geh=C^{(1)}_{n}$} \\
\alpha_0+\alpha_1+2\alpha_2+\cdots+2\alpha_{n-2}
+\alpha_{n-1}+\alpha_n
& \text{if $\geh=D^{(1)}_{n}$} \\
2\alpha_0+2\alpha_1+\cdots+2\alpha_{n-1}+\alpha_n
& \text{if $\geh=A^{(2)}_{2n}$} \\
\alpha_0+\alpha_1+2\alpha_2+\cdots+2\alpha_{n-1}+\alpha_n & \text{if
$\geh=A^{(2)}_{2n-1}$} \\
\alpha_0+\alpha_1+\cdots+\alpha_{n-1}+\alpha_n &\text{if
$\geh=D^{(2)}_{n+1}$}.
\end{cases}
\end{equation}
The dual Kac label $a^\vee_i$ is the label $a_i$ for the affine
Dynkin diagram obtained by reversing the arrows of the Dynkin
diagram of $\geh$. Note that $a_0^\vee=1$.

Let $(\cdot|\cdot)$ be the normalized invariant form on the weight lattice $P$ \cite{Kac}. It satisfies
\begin{equation}
  (\alpha_i|\alpha_j) = \dfrac{a_i^\vee}{a_i} A_{ij}
\end{equation}
for $i,j\in I$. In particular
\begin{equation}
(\alpha_a|\alpha_a)=2r
\end{equation}
if $\alpha_a$ is a long root. For $i \in I$ let
\begin{equation}\label{eq:ttdef}
t_i = \max(\frac{a_i}{a^\vee_i},a^\vee_0), \qquad t^\vee_i =
\max(\frac{a^\vee_i}{a_i},a_0).
\end{equation}
We shall only use $t^\vee_i$ and $t_i$ for $i \in I_0$, where we have set $I_0=I\setminus\{0\}$. 
For $a\in I_0$ we have
\begin{equation*}
t^\vee_a = 1 \,\,\text{ if $r = 1$,} \qquad t_a = 1 \,\,\text{ if $r = 2$}.
\end{equation*}

We consider two finite-dimensional subalgebras of $\geh$: $\gehc$,
whose Dynkin diagram is obtained from that of $\geh$ by removing
the $0$ vertex, and $\gehz$, the subalgebra of $X_N$ fixed by the
automorphism $\sigma$ given in \cite[Section 8.3]{Kac}.
\begin{table}[ht]
\caption{}\label{tab:geh-bar}
\vspace{-1mm}
\begin{center}
\begin{tabular}{c|ccccc}
$\geh$ & $X^{(1)}_N$ & $A^{(2)}_{2n}$ & $A^{(2)}_{2n-1}$ & $D^{(2)}_{n+1}$ \\
\hline
$\gehc$ & $X_N$ & $C_n$ & $C_n$ & $B_n$ \\
$\gehz$ & $X_N$ & $B_n$ & $C_n$ & $B_n$
\end{tabular}
\end{center}
\end{table}
Let $\gehc$ (resp. $\gehz$) have weight lattice $\Pb$ (resp.
$\Pt$), with simple roots and fundamental weights
$\alpha_a,\Lab_a$ (resp. $\alt_a,\Lat_a$) for $a\in I_0$. Note that
$\gehc=\gehz$ for $\geh\not=A^{(2)}_{2n}$. For $\geh=A^{(2)}_{2n}$,
$\gehc=C_n$ and $\gehz=B_n$.
$\Pt$ is endowed with the bilinear form $(\cdot | \cdot)'$,
normalized by
\begin{equation} \label{eq:tilde form norm}
  (\alt_a|\alt_a )' = 2r \qquad\text{if
$\alt_a$ is a long root of $\gehz$.}
\end{equation}
For $A^{(2)}_2$ the unique simple root $\alt_1$ of $\gehz=B_1$ is
considered to be short.
Note that $\alpha_a,\Lab_a$ and $(\cdot | \cdot)$ may be
identified with $\alt_a,\Lat_a$ and $(\cdot | \cdot)'$ if $\geh\not= A^{(2)}_{2n}$.

Define the $\Z$-linear map $\iota : \Pb \rightarrow \Pt$ by
\begin{equation}\label{iota}
\iota(\Lab_a) = \epsilon_a \Lat_a \qquad \text{for $a\in I_0$,}
\end{equation}
where $\epsilon_a$ is defined by
\[
\epsilon_a = \begin{cases}
2 & \text{if $\geh=A^{(2)}_{2n}$ and $a=n$}\\
1 & \text{otherwise.}
\end{cases}
\]
In particular $\iota(\alpha_a) = \epsilon_a \alt_a$ for $a\in I_0$.
If $\geh=A^{(2)}_{2n}$, we have $(\alpha_b | \alpha_b)=4,
(\tilde{\alpha}_b | \tilde{\alpha}_b)'=2$ if $b=n$ and $2,4$ otherwise.
In the rest of the paper we shall write $(\cdot \vert\cdot)$ in
place of $(\cdot|\cdot)'$.

We now associate a kind $\dia$ to each nonexceptional affine algebra $\geh$ as follows.
\begin{table}[ht]
\caption{}\label{tab:diamond}
\vspace{-6mm}
\begin{align*}
\begin{array}{|c||c|} \hline
\dia &\text{$\geh$ of kind $\dia$} \\ \hline \hline
\varnothing & A_n^{(1)} \\ \hline
\cell & D_{n+1}^{(2)}, A_{2n}^{(2)} \\ \hline
\hdom & C_n^{(1)} \\ \hline
\vdom & A_{2n-1}^{(2)}, B_n^{(1)}, D_n^{(1)} \\ \hline
\end{array}
\end{align*}
\end{table}
The kind depends precisely on the attachment of the affine Dynkin node $0$ to the rest of the
Dynkin diagram. We use the labeling of the Dynkin nodes in \cite{Kac}. It is also related to 
the $I_0$-decomposition of the Kirillov-Reshetikhin crystal $B^{r,s}$ for $n$ large with respect to $r$
(see \cite{LOS}).

\subsection{Classical weights, Littlewood-Richardson skew tableaux}
For classical simple Lie algebras $\gehc=A_n,B_n,C_n,D_n$ we often identify the dominant integral
weight without spin with a Young diagram. Namely, for $\la=\sum_{j=1}^{n'}m_j\Lab_j$ ($n'=n$ for
$A_n$,\,$=n-1$ for $B_n,C_n$,\,$=n-2$ for $D_n$, $m_j\in\Z_{\ge0}$, $\Lab_j$ is a fundamental weight
of $\gehc$), we associate the Young diagram such that there are exactly $m_j$ columns of height $j$. 
We utilize this identification throughout this paper. As usual, for a Young diagram (or partition) $\la$,
$\ell(\la)$ denotes its depth and $|\la|$ the number of boxes.

Next we explain the Littlewood-Richardson skew tableau. A skew tableau $T$ is called a 
Littlewood-Richardson skew tableau, LR tableau for short, if it is semi-standard and its reverse 
row word is Yamanouchi, {\it i.e.}, when the word $w=x_1\cdots x_l$ is read from the first to any letter,
the sequence $x_1\cdots x_k$ contains at least as many 1's as it does 2's, at least as many 2's as 3's,
and so on for all positive integers. A skew tableau is said to have weight $\mu=(\mu_1,\ldots,\mu_l)$
if it contains $\mu_1$ 1's, $\mu_2$ 2's, and so on up to $\mu_l$ $l$'s.

\begin{example}
The following skew tableau is an example of LR tableau of shape $(43^31^2)/(2^21)$ and of weight
$(3^22^2)$.
\begin{center}
\unitlength 10pt
\begin{picture}(3,7)
\newcommand{\kuu}{{}}
\put(-0.7,0){\young(\kuu\kuu11,\kuu\kuu2,\kuu13,224,3,4)}
\end{picture}
\end{center}
Its reverse row word is $1123142234$, which satisfies the Yamanouchi condition.
\end{example}

It is well known (see e.g.\,\cite{F}) that the number of LR tableaux of shape $\eta/\la$ and 
of weight $\mu$ is
equal to the Littlewoood-Richardson coefficient $c_{\la\mu}^\eta$ that counts the multiplicity 
of $V_\eta$ in $V_\la\ot V_\mu$, where $V_\la$ is the irreducible $\mathfrak{gl}_n$-module of
highest weight $\la$.

\subsection{Rigged configurations}
We recall the fermionic formula given in \cite{HKOTY,HKOTT}.
Let $\Pb^+$ be the set of dominant integral weights of $\gehc$.
Fix $\la\in\Pb^+$ and a matrix
$\El=(L_i^{(a)})_{a\in I_0,i\in\Z_{>0}}$ of nonnegative integers with finitely many positive ones.
Let $\nu=(m_i^{(a)})_{a\in I_0,i\in\Z_{>0}}$ be another matrix of nonnegative integers. 
Say that $\nu$ is a $\la$-configuration if
\begin{equation} \label{config} 
\sum_{a\in I_0,i\in\Z_{>0}} i\,
m_i^{(a)} \alt_a = \iota\Biggl( \sum_{a\in I_0,i\in\Z_{>0}} i \,L_i^{(a)} \Lab_a - \la \Biggr).
\end{equation}
This is equivalent to assuming 
\begin{equation} \label{num of box}
\sum_{i\in\Z_{>0}}im_i^{(a)}=\frac2{(\alt_a|\alt_a)}\sum_{b\in I_0}
\epsilon_b\Biggl(\sum_{i\in\Z_{>0}}iL_i^{(b)}-\langle\la,h_b\rangle\Biggr)(\Lat_a|\Lat_b)
\end{equation}
for $a\in I_0$, where $h_b$ is a simple coroot of $\gehc$.
Say that a configuration $\nu$ is $\El$-admissible if
\begin{equation} \label{ppos}
  p_i^{(a)} \ge 0\qquad\text{for all $a\in I_0$ and
  $i\in\Z_{>0}$,}
\end{equation}
where
\begin{equation} \label{p}
p_i^{(a)} = \sum_{k\in\Z_{>0}} \left( L_k^{(a)} \min(i,k) -
\dfrac{1}{t_a^\vee} \sum_{b\in I_0} (\alt_a|\alt_b)\min(t_b i,t_a
k)\, m_k^{(b)}\right).
\end{equation}
Write $C(\la,\El)$ for the set of $\El$-admissible
$\la$-configurations. Define
\begin{align}
c(\nu) &= \dfrac{1}{2} \sum_{a,b\in I_0} \sum_{j,k\in\Z_{>0}} (\alt_a|\alt_b)
\min(t_b j, t_a k) m_j^{(a)} m_k^{(b)} \label{c}\\
&\hspace{2cm} -\sum_{a\in I_0}t_a^\vee \sum_{j,k\in\Z_{>0}}\min(j,k)L_j^{(a)}m_k^{(a)} \nonumber.
\end{align}
The fermionic formula is defined by
\begin{equation}\label{fermi}
M(\la,\El;q) = \sum_{\nu\in C(\la,\El)} q^{c(\nu)}
\prod_{a\in I_0} \prod_{i\in\Z_{>0}}
\qbin{p_i^{(a)}+m_i^{(a)}}{m_i^{(a)}}_{q^{t^\vee_a}}.
\end{equation}

The fermionic formula $M(\la,\El)$ can be interpreted using
combinatorial objects called rigged configurations. For $a\in I_0$, define
\begin{equation} \label{upsilon}
  \upsilon_a = \begin{cases}
  2 & \text{if $a=n$ and $\geh=C_n^{(1)}$} \\
  \frac{1}{2} & \text{if $a=n$ and $\geh=B_n^{(1)}$} \\
  1 & \text{otherwise.}
  \end{cases}
\end{equation}
$\upsilon_a$ is half the square length of $\alpha_a$ for untwisted
affine types and is equal to $1$ for twisted types.

A quasipartition of type $a\in I_0$ is a finite multiset taken
from the set $\upsilon_a \Z_{>0}$. The diagram of a quasipartition has
rows consisting of boxes with width $\upsilon_a$.
Denote by $(\nu^\bullet,J^\bullet)$ a pair where $\nu^\bullet=\{\nu^{(a)}\}_{a\in I_0}$
is a sequence of quasipartitions with $\nu^{(a)}$ of type $a$ such that
$\nu^{(a)}=(\upsilon_a^{m_1^{(a)}},(2\upsilon_a)^{m_2^{(a)}},\ldots)$ and
$J=\{J^{(a,i)}\}_{a\in I_0,i\in\Z_{>0}}$ is a double sequence of partitions. 
Then a rigged configuration is a pair $(\nu^\bullet,J^\bullet)$ subject to the
restriction \eqref{config} and the requirement that $J^{(a,i)}$
be a partition contained in a $m_i^{(a)}\!\times p_i^{(a)}$ rectangle.
$\nu^\bullet$ is called a configuration and $J^\bullet$ a rigging.
The set of rigged configurations for
fixed $\la$ and $\El$ is denoted by $\mathrm{RC}(\la,\El)$. Then
\eqref{fermi} is equivalent to
\begin{equation*}
M(\la,\El;q)=\sum_{(\nu^\bullet,J^\bullet)\in\mathrm{RC}(\la,\El)} q^{c(\nu^\bullet,J^\bullet)}
\end{equation*}
where $c(\nu^\bullet,J^\bullet)=c(\nu)+|J^\bullet|$
and $|J^\bullet|=\sum_{a\in I_0,i\in\Z_{>0}} t_a^\vee |J^{(a,i)}|$.

For a quasipartition $\mu=(\upsilon^{m_1},(2\upsilon)^{m_2},\ldots)$ with boxes of width $\upsilon$
and $i\in\upsilon\Z_{>0}$, define
\begin{equation} \label{Qdef}
Q_i(\mu)=\sum_k \min(i,\upsilon k)m_k,
\end{equation}
the area of $\mu$ in the first $i/\upsilon$ columns of width $\upsilon$. We also set 
\begin{equation} \label{Qmax}
Q_{\max}(\mu)=\upsilon\sum_k km_k.
\end{equation}
For a configuration $\nu^\bullet=\{\nu^{(a)}\}_{a\in I_0}$ we set $Q_i^{(a)}=Q_i(\nu^{(a)}),
Q_{\max}^{(a)}=Q_{\max}(\nu^{(a)})$.
Then the vacancy numbers $p_i^{(a)}$ are given by
\begin{equation} \label{vacancy}
p_i^{(a)}=\sum_{k\in\Z_{>0}}L_k^{(a)}\min(i,k)+Q_i^{(a-1)}-2Q_i^{(a)}+Q_i^{(a+1)}
\end{equation}
except the ones below in each nonexceptional affine type. 
In the above formula $\nu^{(0)}$ should be considered to be empty.

$A_n^{(1)}$:
\[
p_i^{(n)}=Q_i^{(n-1)}-2Q_i^{(n)}
\]

$B_n^{(1)}$:
\begin{align*}
p_i^{(n-1)}&=Q_i^{(n-2)}-2Q_i^{(n-1)}+2Q_i^{(n)}\\
p_i^{(n)}&=2Q_{i/2}^{(n-1)}-4Q_{i/2}^{(n)}
\end{align*}

$C_n^{(1)}$:
\[
p_i^{(n)}=Q_{2i}^{(n-1)}-Q_{2i}^{(n)}
\]

$D_n^{(1)}$:
\begin{align*}
p_i^{(n-2)}&=Q_i^{(n-3)}-2Q_i^{(n-2)}+Q_i^{(n-1)}+Q_i^{(n)}\\
p_i^{(n-1)}&=Q_i^{(n-2)}-2Q_i^{(n-1)}\\
p_i^{(n)}&=Q_i^{(n-2)}-2Q_i^{(n)}
\end{align*}

$A_{2n}^{(2)}$:
\[
p_i^{(n)}=Q_i^{(n-1)}-Q_i^{(n)}
\]

$A_{2n-1}^{(2)}$:
\begin{align*}
p_i^{(n-1)}&=Q_i^{(n-2)}-2Q_i^{(n-1)}+2Q_i^{(n)}\\
p_i^{(n)}&=Q_i^{(n-1)}-2Q_i^{(n)}
\end{align*}

$D_{n+1}^{(2)}$:
\[
p_i^{(n)}=2Q_i^{(n-1)}-2Q_i^{(n)}
\]
We assumed $L_i^{(a)}=0$ for any pair $(a,i)$ that $p_i^{(a)}$ appears in the above list, 
since it is enough for our calculations later.

We show $p_i^{(a)}\ge0$ for all $i$ such that $m_i^{(a)}>0$ implies $p_i^{(a)}\ge0$ for all the other $i$.
\begin{lemma}\label{lem:convex}
Suppose $m_i^{(a)}=0$ for $k<i<l$. Then the vacancy numbers satisfy the following upper convex relation:
\[
p^{(a)}_{k}+p^{(a)}_{l}\leq
2p^{(a)}_i
\]
for $k<i<l$.
\end{lemma}
\begin{proof}
Suppose the vacancy number is given by \eqref{vacancy}.
On the interval $[k,l]$, $Q^{(a)}_i$ is a linear function of $i$
since there is no length $i$ row of $\nu^{(a)}$ for $k<i<l$.
On the other hand, $Q^{(a-1)}_i$ and $Q^{(a+1)}_i$ are upper convex functions of $i$.
Similarly, $\sum_kL_k^{(a)}\min(i,k)$ is also an upper convex function of $i$. Hence,
we obtain the result.
Proof for other $a$'s are the same.
\end{proof}

\begin{cor}\label{cor:positivity}
Suppose $p_i^{(a)}\ge0$ for all $i$ such that $m_i^{(a)}>0$. 
Then we have $p^{(a)}_i\ge0$ for $i\in\Z_{>0}$.
\end{cor}
\begin{proof}
If $i\leq \nu^{(a)}_1$ we obtain the result
by Lemma \ref{lem:convex}. (Note that $k$ can be $0$.)
On the contrary, assume $\nu^{(a)}_1<i$.
Then apply Lemma \ref{lem:convex} with $k=\nu^{(a)}_1,l\gg\nu^{(a)}_1$.
\end{proof}

\subsection{Stability} \label{subsec:stability}
In this subsection we investigate the behavior of a rigged configuration when $n$ is large.
We define an integer $\gamma$ by $\gamma=2$ for $\geh=A^{(2)}_{2n},D^{(2)}_{n+1}$, $\gamma=1$ otherwise.

\begin{lemma} \label{lem:stability}
Let $(\nu^\bullet,J^\bullet)\in\mathrm{RC}(\la,\El)$ and 
$k=\max\{a\mid\langle\la,h_a\rangle\ne0\text{ or }L_i^{(a)}\ne0\text{ for some }i\}$.
Then we have
\begin{align*}
Q_{\max}^{(k)}=Q_{\max}^{(k+1)}&=\cdots=Q_{\max}^{(n)}=0
\text{ for }\geh=A_n^{(1)},\\
Q_{\max}^{(k)}=Q_{\max}^{(k+1)}&=\cdots=Q_{\max}^{(n-1)}=2Q_{\max}^{(n)}
\text{ for }\geh=A_{2n-1}^{(2)},B_n^{(1)},\\
Q_{\max}^{(k)}=Q_{\max}^{(k+1)}&=\cdots=2Q_{\max}^{(n-1)}=2Q_{\max}^{(n)}
\text{ for }\geh=D_n^{(1)},\\
Q_{\max}^{(k)}=Q_{\max}^{(k+1)}&=\cdots=Q_{\max}^{(n-1)}=Q_{\max}^{(n)}
\text{ otherwise}.
\end{align*}
\end{lemma}

\begin{proof}
Look at the formula \eqref{num of box}.
The definition of $k$ restricts the summation over $b\in I_0$ to $1\le b\le k$. Recalling
\eqref{Qmax} with \eqref{upsilon} we obtain the desired result.
\end{proof}

\begin{lemma} \label{lem:width}
Let $k$ be as in Lemma \ref{lem:stability}.
Then we have $\nu^{(k+1)}_1=\nu^{(k+2)}_1=\cdots =\nu^{(n)}_1$.
In particular, the longest rows of $\nu^{(k+1)}, \nu^{(k+2)},
\ldots, \nu^{(n)}$ are singular.
\end{lemma}

\begin{proof}
We give the proof for $\geh=D_n^{(1)}$.
Proofs for the other cases are similar.
To begin with note that 
$Q_{\max}^{(k)}=Q_{\max}^{(k+1)}=\cdots=Q_{\max}^{(n-2)}=2Q_{\max}^{(n-1)}=2Q_{\max}^{(n)}$
from Lemma \ref{lem:stability}.
Assume that $\nu^{(n-2)}_1>\nu^{(n)}_1$.
Then we have
$p^{(n)}_{\nu^{(n)}_1}=Q^{(n-2)}_{\nu^{(n)}_1}-2Q^{(n)}_{\nu^{(n)}_1}
<Q^{(n-2)}_{\nu^{(n-2)}_1}-2Q^{(n)}_{\nu^{(n)}_1}=0$,
which is a contradiction.
Therefore it has to be $\nu^{(n-2)}_1\leq \nu^{(n)}_1$.
Similarly we have $\nu^{(n-2)}_1\leq \nu^{(n-1)}_1$.
Suppose that one of the inequalities is strict, say
$\nu^{(n-2)}_1<\nu^{(n)}_1$.
Then we have
$p^{(n-2)}_{\nu^{(n-2)}_1}=
Q^{(n-3)}_{\nu^{(n-2)}_1}-2Q^{(n-2)}_{\nu^{(n-2)}_1}
+Q^{(n-1)}_{\nu^{(n-2)}_1}+Q^{(n)}_{\nu^{(n-2)}_1}<
Q^{(n-3)}_{\nu^{(n-2)}_1}-2Q^{(n-2)}_{\nu^{(n-2)}_1}
+Q^{(n-1)}_{\nu^{(n-1)}_1}+Q^{(n)}_{\nu^{(n)}_1}\leq 0$
(note that $Q^{(n-3)}_{\nu^{(n-2)}_1}\leq
Q^{(n-3)}_{\nu^{(n-3)}_1}$),
which is a contradiction.
Hence, we have
$\nu^{(n-2)}_1=\nu^{(n-1)}_1=\nu^{(n)}_1$.
We can recursively do the same argument up to $\nu^{(k+1)}$
and finally we obtain
$\nu^{(k)}_1\leq \nu^{(k+1)}_1=\nu^{(k+2)}_1=\cdots =\nu^{(n)}_1$.
The statement about singularity also follows from this estimate.
\end{proof}

\begin{example}
The above lemma assures the singularity for the longest rows.
If the row is not longest, there could be non-singular rows.
For example, for type $D^{(1)}_6$, there is the
following example
($\la=(3),L_1^{(1)}=15,L_i^{(a)}=0$ for other $(a,i)$, i.e., $k=1$).
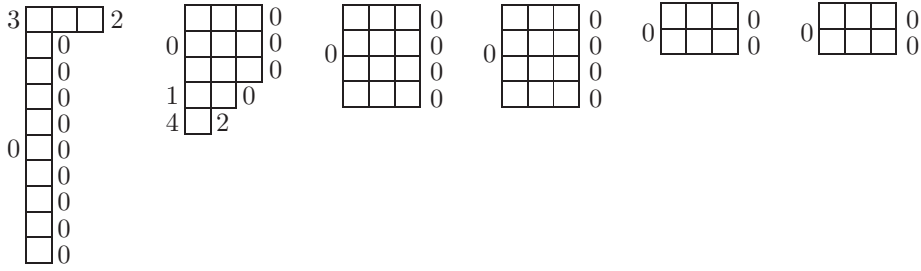
\begin{center}
\unitlength 10pt
\begin{picture}(33,11)
\Yboxdim{10pt}
\put(0,0.1){\yng(3,1,1,1,1,1,1,1,1,1)}
\put(-0.7,4.1){0}
\put(-0.7,9.0){3}
\put(3.2,9.0){2}
\multiput(1.2,0.1)(0,0.99){9}{0}
\put(6,5){\yng(3,3,3,2,1)}
\put(5.3,5.1){4}
\put(5.3,6.1){1}
\put(5.3,8){0}
\put(7.2,5.1){2}
\put(8.2,6.1){0}
\multiput(9.2,7.1)(0,1){3}{0}
\put(12,6){\yng(3,3,3,3)}
\put(11.3,7.7){0}
\multiput(15.3,6)(0,1){4}{0}
\put(18,6){\yng(3,3,3,3)}
\put(17.3,7.7){0}
\multiput(21.3,6)(0,1){4}{0}
\put(24,8){\yng(3,3)}
\put(23.3,8.5){0}
\multiput(27.3,8)(0,1){2}{0}
\put(30,8){\yng(3,3)}
\put(29.3,8.5){0}
\multiput(33.3,8)(0,1){2}{0}
\end{picture}
\end{center}
The length 3 rows of $\nu^{(a)}$ ($2\leq a$) are singular.
For the reader's convenience, we record here
the corresponding tensor product:
$\Yvcentermath1
\Yboxdim 12pt
\newcommand{\bone}{\bar{1}}
\newcommand{\btwo}{\bar{2}}
\newcommand{\bthree}{\bar{3}}
\young(1)\otimes\young(2)\otimes\young(1)
\otimes\young(2)\otimes\young(3)\otimes
\young(\btwo)\otimes\young(1)\otimes\young(\bthree)
\otimes\young(\btwo)\otimes
\young(1)\otimes\young(2)\otimes
\young(\btwo)\otimes\young(\bone)\otimes
\young(2)\otimes\young(\btwo)\in (B^{1,1})^{\otimes 15}$.
\end{example}

\begin{prop} \label{prop:stability}
Let $(\nu^\bullet,J^\bullet)\in\mathrm{RC}(\la,\El)$ and $k$ as in Lemma \ref{lem:stability}.
Then there exists a partition $\nu^*=(1^{m_1^*},2^{m_2^*},\ldots)$ tiled by $\dia$
such that by setting $l^*=k+\ell(\nu^*)$
\begin{itemize}
\item[(1)] $\nu^{(a)}=\nu^*$ for $a\ge l^*$
	except when $a=n$ for $\geh=A_{2n-1}^{(2)},B_n^{(1)}$ and $a=n-1,n$ for $\geh=D_n^{(1)}$,
	in which case $\nu^{(a)}$ is given by halving each column of $\nu^*$,
\item[(2)] $p_i^{(a)}=0$ for any $i$ and $a>l^*$,
\item[(3)] $c(\nu)=\frac{\gamma}2\sum_{a,b\le l^* \atop j,k\in\Z_{>0}}
	C_{ab}\min(j,k)m_j^{(a)}m_k^{(b)}
	-\gamma\sum_{a\le l^* \atop j,k\in\Z_{>0}}\min(j,k)L_j^{(a)}m_k^{(a)}$ where 
	$C_{ab}=(2-\delta_{al^*})\delta_{ab}-\delta_{a,b-1}-\delta_{a,b+1}$.
\end{itemize}
\end{prop}

\begin{proof}
(1) For the case $\dia=\varnothing$, that is $\geh=A_n^{(1)}$, $Q_{\max}^{(a)}=0$ for
$a\ge k$ from Lemma \ref{lem:stability}. Hence $\nu^{(a)}=\emptyset$ for $a\ge l^*$. The proof
for the other cases are more or less the same. We only consider the case $\geh=C_n^{(1)}$.

Let $w_a$ be the length of the longest row of $\nu^{(a)}$. Note that $w_n$ is even. 
Lemma \ref{lem:width} shows $w_a=w_{a+1}$ for $a\ge k+1$.
Let us show the length $h_a$ of the rightmost ($w$-th) column of $\nu^{(a)}$ is equal for $a\ge k+h_n$.
Let $m$ be the minimal integer such that $m\ge k+1,h_m=h_{m+1}=\cdots=h_n$. Then $h_{m-1}<h_m$ 
since $p_{w-1}^{(m)}=Q_{w-1}^{(m-1)}-2Q_{w-1}^{(m)}+Q_{w-1}^{(m+1)}=Q_{w-1}^{(m-1)}-Q_{w-1}^{(m)}\ge0$,
where we have used $h_m=h_{m+1}$ in the second equality.
{}From $p_{w-1}^{(m-1)}\ge0$ we have $h_{m-2}-2h_{m-1}+h_m\le0$. Hence we get $h_{m-2}<h_{m-1}$.
This argument continues until we get $h_{k+1}<h_{k+2}<\cdots<h_m$. It implies $m-k\le h_n$.
Therefore we have $h_a=h_{a+1}$ for $a\ge k+h_n$.

Let $h'_a$ be the length of the next rightmost ($(w-1)$-th) column of $\nu^{(a)}$. A similar argument
to the previous paragraph shows that $h'_a=h'_{a+1}$ for $a\ge k+h_n+h'_n$ and so on, 
arriving at the conclusion that $\nu^{(a)}=\nu^{(a+1)}$ for $a\ge k+\ell(\nu^{(n)})$.
The fact that $\nu^*$ is tiled by $\hdom$ is a 
consequence from the fact that $\nu^{(n)}$ is a quasipartition with boxes of width 2.

(2) is clear from (1) and the fact that $p_i^{(a)}$ is calculated by $\nu^{(a-1)},\nu^{(a)},\nu^{(a+1)}$.
(3) can be proven using (1) by case-by-case checking.
\end{proof}

{}From Proposition \ref{prop:stability} we see if $n$ is sufficiently large, the fermionic 
formula \eqref{fermi} can be rewritten as
\begin{align*}
M(\la,\El;q) &= \sum_{\nu\in C(\la,\El)} q^{c(\nu)}
\prod_{a\le l^*} \prod_{i\in\Z_{>0}}
\qbin{p_i^{(a)}+m_i^{(a)}}{m_i^{(a)}}_{q^\gamma} \\
&= \sum_{(\nu^\bullet,J^\bullet)\in\mathrm{RC}(\la,\El)} q^{c(\nu^\bullet,J^\bullet)}
\end{align*}
where $c(\nu^\bullet,J^\bullet)=c(\nu)+|J^\bullet|$,
$|J^\bullet|=\gamma\sum_{a\le l^*,i\in\Z_{>0}}|J^{(a,i)}|$ and $l^*$ is as in Proposition 
\ref{prop:stability}. Apparently, it depends only on the symbol $\dia$, but not on the affine algebra
belonging to the same group in Table \ref{tab:diamond}. Hence, from now on we pick up one 
affine algebra $\geh^\dia=A_n^{(1)},D_{n+1}^{(2)},C_n^{(1)},D_n^{(1)}$ for each symbol 
$\dia=\varnothing,\cell,\hdom,\vdom$, and define the set of stable rigged configurations 
$\mathrm{RC}^\dia$ and stable fermionic formula $M^{\dia}(\la,\El;q)$
by using rigged configurations of $\geh^\dia$ for $n$ large. Therefore, we have
\[
M^\dia(\la,\El;q)=\sum_{(\nu^\bullet,J^\bullet)\in\mathrm{RC}^\dia(\la,\El)} q^{c(\nu^\bullet,J^\bullet)}.
\]
\begin{remark}
The above choice of the affine algebra $\geh^\dia$ for each $\dia$ is not mandatory. Namely, the 
construction of our main bijection $\Psi$ in the subsequent section or its properties are 
essentially the same even if we choose a different $\geh^\dia$. Compare this situation with that
in \cite{LOS} where the choice is essential.
\end{remark}

For a stable rigged configuration $(\nu^\bullet,J^\bullet)$ one can calculate from \eqref{num of box}
$\la$ as
\[
\la_a=\sum_{b\ge a,i\in\Z_{>0}}iL_i^{(b)}+|\nu^{(a-1)}|-|\nu^{(a)}|,
\]
where $\la_a$ is the length of the $a$-th row of $\la$ when identified with the Young diagram. For
$(\nu^\bullet,J^\bullet)\in\mathrm{RC}^\dia(\la,\El)$ we denote it by $\wt(\nu^\bullet,J^\bullet)$.
We note that $k$ in Lemma \ref{lem:stability} is equal to $\ell(\wt(\nu^\bullet,J^\bullet))$.

\section{The bijection} \label{sec:bijection}
\subsection{Definitions}
The goal of this subsection is to give definitions
of our main algorithms $\Psi$ and $\tilde{\Psi}$.
Roughly speaking, the algorithms consist of two parts:
the one is box removing or adding procedure on the
rigged configurations, and the other one is to create
a kind of recording tableau $T$ which eventually generates
the LR tableaux.
We will divide the definition according to
this distinction.
During the algorithm, we will use
$(\nu^\bullet,J^\bullet)$ in slightly generalized sense.
More precisely, for the case $\diamondsuit =\vdom$,
we allow the vacancy number for the longest row of $\nu^{(n)}$
to be $-1$ while its rigging is $0$.
Such a peculiar situation happens only if we consider
odd powers of the operators $\delta$ or $\tilde{\delta}$
(see definitions below) for $\diamondsuit =\vdom$, though the final
algorithms $\Psi$ and $\tilde{\Psi}$ always contain
even powers of such operators whenever $\diamondsuit =\vdom$.
In this section we use the symbol $a^\diamondsuit$
where $a^\hdom =a^\sdom =n-1$
and $a^\vdom =n-2$.

\begin{definition}\label{def:delta}
The map $\delta_l$
\[\delta_l :
(\nu^\bullet,J^\bullet)
\longmapsto
\{(\nu'^{\bullet},J'^\bullet),k\},
\]
is defined
by the following algorithm.
Here $l$ is the length of some row of $\nu^{(n)}$.
\begin{enumerate}
\item[(i)] As the initial step, do one of the following:
\begin{enumerate}
\item If $\diamondsuit =\hdom$ or $\sdom$,
choose one of the length $l$ rows
of $\nu^{(n)}$.
\item The case $\diamondsuit =\vdom$.
If $|\nu^{(n-1)}|=|\nu^{(n)}|$, choose
one of the length $l$ rows of $\nu^{(n-1)}$.
On the other hand, if $|\nu^{(n-1)}|<|\nu^{(n)}|$,
choose one of length $l$
rows of $\nu^{(n)}$.
\end{enumerate}
\item[(ii)] 
Choose one of the length $l$ rows of $\nu^{(a^\diamondsuit)}$.
Then choose rows of $\nu^{(a)}$ $(a<a^\diamondsuit)$ recursively as follows.
Suppose that we have chosen a row of $\nu^{(a)}$.
Choose a shortest singular row of $\nu^{(a-1)}$ whose length is equal to
or longer than the chosen row of $\nu^{(a)}$, and continue, if such a row exists;
otherwise (in particular when $a=1$) set $k=a$ and stop.
\item[(iii)]
$\nu'^\bullet$ is obtained by removing one box from
the right end of each chosen row at Step (i) and (ii).
\item[(iv)]
The new riggings $J'^\bullet$ are defined as follows.
For the rows that are not changed in Step (iii),
take the same riggings as before.
Otherwise set the new riggings equal to the corresponding
vacancy numbers computed by using $\nu'^\bullet$.
\end{enumerate}
\end{definition}

\begin{definition}\label{def:Psi}
The map $\Psi$
\[
\Psi : (\nu^\bullet,J^\bullet)\longmapsto
\{(\nu'^\bullet,J'^\bullet),T\}
\]
is defined as follows.
As the initial condition, set $T$ to be the empty skew tableau with both inner 
and outer shape equal to the diagram representing the weight of
$(\nu^\bullet,J^\bullet)$.
Let $h_i$ denote the height of the $i$-th column (counting
from left) of the partition $\nu^{(a^\diamondsuit)}$
and let $l=\nu^{(a^\diamondsuit)}_1$.

\begin{enumerate}
\item[(i)]
We will apply $\delta_{l}$
for $h_{l}$ times.
Each time when we apply $\delta_{l}$,
we recursively redefine $(\nu^\bullet,J^\bullet)$
and $T$ as follows.
Assume that we have done $\delta_{l}^{i-1}$
and obtained $\{(\nu^\bullet,J^\bullet),T\}$.
Let us apply $\delta_{l}$ one more time:
\[
\delta_l :
(\nu^\bullet,J^\bullet)
\longmapsto
\{(\nu'^{\bullet},J'^\bullet),k\},
\]
Using the output, do the following.
Define the new $(\nu^\bullet,J^\bullet)$ to be
$(\nu'^{\bullet},J'^\bullet)$.
Define the new $T$ by putting $i$ on the right of
the $k$-th row of the previous $T$.
\item[(ii)]
Recursively apply
$\delta_{l-1}^{h_{l-1}},
\ldots,\delta_2^{h_2},\delta_1^{h_1}$
by the same procedure as in Step (i).
Then the final outputs
$(\nu'^\bullet,J'^\bullet)$ and $T$
give the image of $\Psi$.
\end{enumerate}
\end{definition}

\begin{example}\label{ex:Psi}
Let us consider the special case of the bijection $\Psi$
where the bijection \cite{OSS,SS} between the rigged configurations
and the tensor products of crystals is also available.
Consider the following element of the tensor product
$(B^{1,3})^{\otimes 3}\otimes
(B^{1,2})^{\otimes 2}\otimes
(B^{1,1})^{\otimes 2}$
of type $D^{(1)}_n$ ($n\geq 8$) crystals:
\[
\newcommand{\bone}{\bar{1}}
\newcommand{\btwo}{\bar{2}}
\Yboxdim 12pt
\Yvcentermath1
p=
\young(111)\otimes\young(2\bone\bone)\otimes
\young(12\btwo)\otimes\young(23)\otimes
\young(2\btwo)\otimes\young(\btwo)\otimes\young(2).
\]
Due to Theorem 8.6 of \cite{SS} all the isomorphic
elements under the combinatorial $R$-matrices correspond
to the same rigged configuration.
Then the map $\Psi$ for the $D^{(1)}_8$ case proceeds as follows.
We remark that one can slightly modify the definition of $\Psi$
so that the following computation can be done in $D^{(1)}_7$.
In the following diagrams,
the first rigged configuration corresponds to the above $p$.
Here, we put the vacancy numbers (resp. riggings)
on the left (resp. right) of the corresponding rows.
The gray boxes represent the boxes to be removed by
each $\delta$ indicated on the left of each arrow.
The corresponding recording tableau $T$ is given
on the right of each arrow.
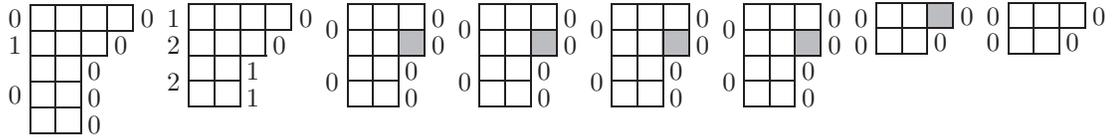
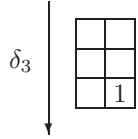
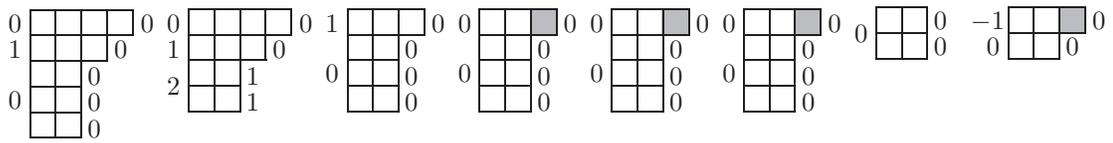
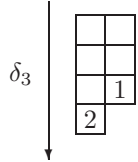
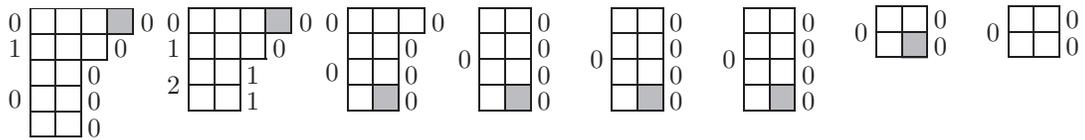
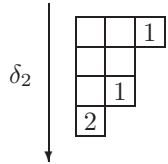
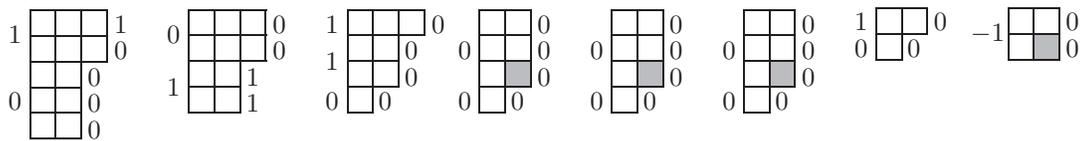
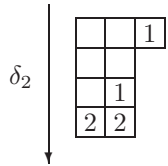
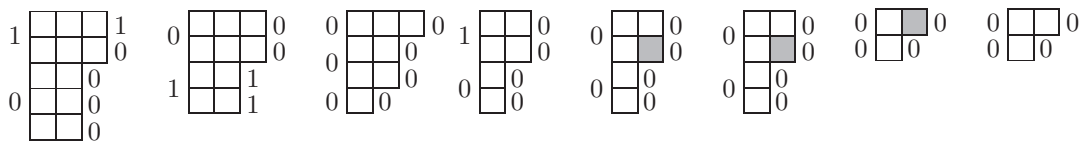
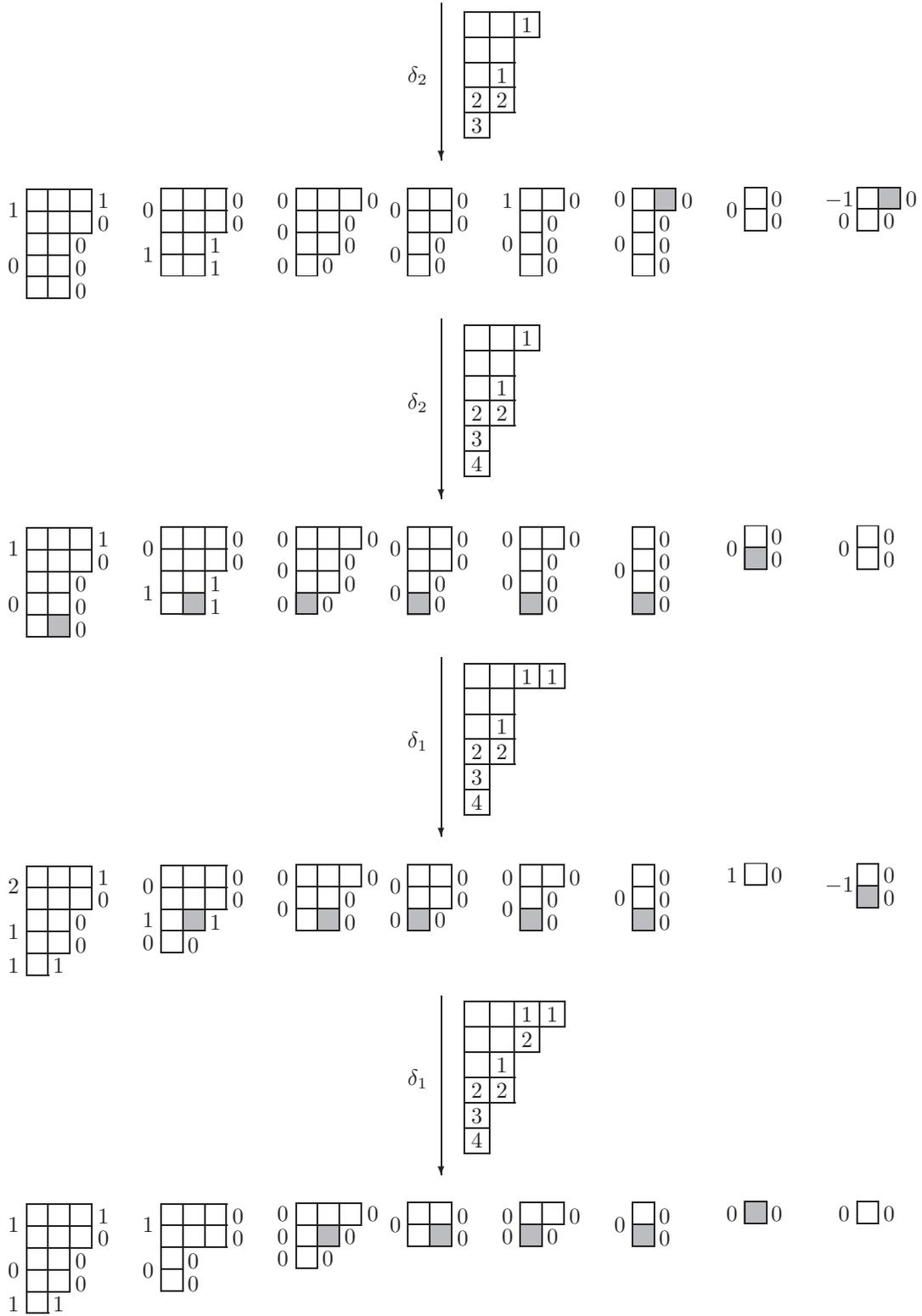
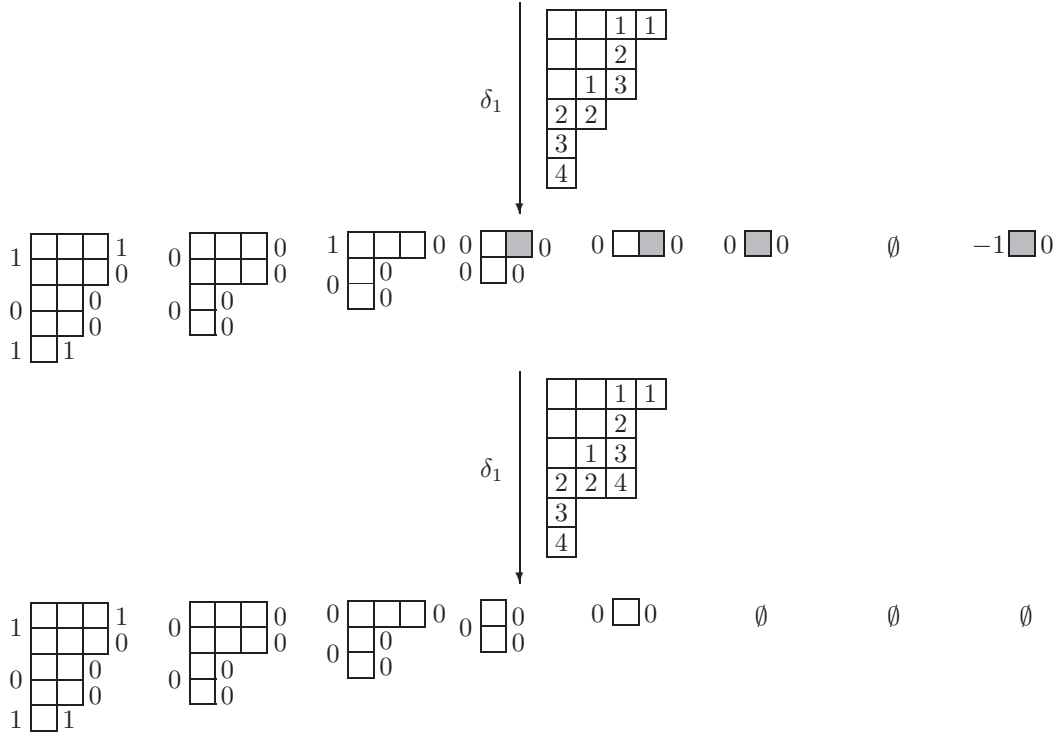
\begin{center}
\unitlength 10pt
\begin{picture}(40,5)(2,0)
\put(14.9,2.9){\gnode}
\put(19.9,2.9){\gnode}
\put(24.9,2.9){\gnode}
\put(29.9,2.9){\gnode}
\put(34.9,3.95){\gnode}
\Yboxdim{10pt}
\put(1,0){\yng(4,3,2,2,2)}
\put(0.2,1.1){0}
\put(0.2,3){1}
\put(0.2,4){0}
\put(3.2,0){0}
\put(3.2,1){0}
\put(3.2,2){0}
\put(4.2,3){0}
\put(5.2,4){0}
\put(7,1){\yng(4,3,2,2)}
\put(6.2,1.6){2}
\put(6.2,3){2}
\put(6.2,4){1}
\put(9.2,1){1}
\put(9.2,2){1}
\put(10.2,3){0}
\put(11.2,4){0}
\put(13,1){\yng(3,3,2,2)}
\put(12.2,1.6){0}
\put(12.2,3.6){0}
\put(15.2,1){0}
\put(15.2,2){0}
\put(16.2,3){0}
\put(16.2,4){0}
\put(18,1){\yng(3,3,2,2)}
\put(17.2,1.6){0}
\put(17.2,3.6){0}
\put(20.2,1){0}
\put(20.2,2){0}
\put(21.2,3){0}
\put(21.2,4){0}
\put(23,1){\yng(3,3,2,2)}
\put(22.2,1.6){0}
\put(22.2,3.6){0}
\put(25.2,1){0}
\put(25.2,2){0}
\put(26.2,3){0}
\put(26.2,4){0}
\put(28,1){\yng(3,3,2,2)}
\put(27.2,1.6){0}
\put(27.2,3.6){0}
\put(30.2,1){0}
\put(30.2,2){0}
\put(31.2,3){0}
\put(31.2,4){0}
\put(33,3){\yng(3,2)}
\put(32.2,3){0}
\put(32.2,4){0}
\put(35.2,3.1){0}
\put(36.2,4.1){0}
\put(38,3){\yng(3,2)}
\put(37.2,3.1){0}
\put(37.2,4.1){0}
\put(40.2,3.1){0}
\put(41.2,4.1){0}
\end{picture}
\end{center}

\begin{center}
\unitlength 10pt
\begin{picture}(3,4)
\newcommand{\kuu}{{}}
\put(-0.5,1.5){$\delta_3$}
\put(2,0){\young(\kuu\kuu,\kuu\kuu,\kuu1)}
\put(1,4){\vector(0,-1){5}}
\end{picture}
\end{center}

\begin{center}
\unitlength 10pt
\begin{picture}(40,5)(2,0)
\put(19.9,3.9){\gnode}
\put(24.9,3.9){\gnode}
\put(29.9,3.9){\gnode}
\put(39.9,3.95){\gnode}
\Yboxdim{10pt}
\put(1,0){\yng(4,3,2,2,2)}
\put(0.2,1.1){0}
\put(0.2,3){1}
\put(0.2,4){0}
\put(3.2,0){0}
\put(3.2,1){0}
\put(3.2,2){0}
\put(4.2,3){0}
\put(5.2,4){0}
\put(7,1){\yng(4,3,2,2)}
\put(6.2,1.6){2}
\put(6.2,3){1}
\put(6.2,4){0}
\put(9.2,1){1}
\put(9.2,2){1}
\put(10.2,3){0}
\put(11.2,4){0}
\put(13,1){\yng(3,2,2,2)}
\put(12.2,2.1){0}
\put(12.2,4){1}
\put(15.2,1){0}
\put(15.2,2){0}
\put(15.2,3){0}
\put(16.2,4){0}
\put(18,1){\yng(3,2,2,2)}
\put(17.2,2.1){0}
\put(17.2,4){0}
\put(20.2,1){0}
\put(20.2,2){0}
\put(20.2,3){0}
\put(21.2,4){0}
\put(23,1){\yng(3,2,2,2)}
\put(22.2,2.1){0}
\put(22.2,4){0}
\put(25.2,1){0}
\put(25.2,2){0}
\put(25.2,3){0}
\put(26.2,4){0}
\put(28,1){\yng(3,2,2,2)}
\put(27.2,2.1){0}
\put(27.2,4){0}
\put(30.2,1){0}
\put(30.2,2){0}
\put(30.2,3){0}
\put(31.2,4){0}
\put(33,3){\yng(2,2)}
\put(32.2,3.6){0}
\put(35.2,3.1){0}
\put(35.2,4.1){0}
\put(38,3){\yng(3,2)}
\put(37.2,3.1){0}
\put(36.6,4.1){$-1$}
\put(40.2,3.1){0}
\put(41.2,4.1){0}
\end{picture}
\end{center}

\begin{center}
\unitlength 10pt
\begin{picture}(3,5)
\newcommand{\kuu}{{}}
\put(-0.5,2){$\delta_3$}
\put(2,0){\young(\kuu\kuu,\kuu\kuu,\kuu1,2)}
\put(1,5){\vector(0,-1){6}}
\end{picture}
\end{center}

\begin{center}
\unitlength 10pt
\begin{picture}(40,5)(2,0)
\put(3.9,3.85){\gnode}
\put(9.9,3.9){\gnode}
\put(13.95,0.95){\gnode}
\put(18.95,0.96){\gnode}
\put(23.95,0.96){\gnode}
\put(28.95,0.96){\gnode}
\put(33.95,2.96){\gnode}
\Yboxdim{10pt}
\put(1,0){\yng(4,3,2,2,2)}
\put(0.2,1.1){0}
\put(0.2,3){1}
\put(0.2,4){0}
\put(3.2,0){0}
\put(3.2,1){0}
\put(3.2,2){0}
\put(4.2,3){0}
\put(5.2,4){0}
\put(7,1){\yng(4,3,2,2)}
\put(6.2,1.6){2}
\put(6.2,3){1}
\put(6.2,4){0}
\put(9.2,1){1}
\put(9.2,2){1}
\put(10.2,3){0}
\put(11.2,4){0}
\put(13,1){\yng(3,2,2,2)}
\put(12.2,2.1){0}
\put(12.2,4){0}
\put(15.2,1){0}
\put(15.2,2){0}
\put(15.2,3){0}
\put(16.2,4){0}
\put(18,1){\yng(2,2,2,2)}
\put(17.2,2.6){0}
\put(20.2,1){0}
\put(20.2,2){0}
\put(20.2,3){0}
\put(20.2,4){0}
\put(23,1){\yng(2,2,2,2)}
\put(22.2,2.6){0}
\put(25.2,1){0}
\put(25.2,2){0}
\put(25.2,3){0}
\put(25.2,4){0}
\put(28,1){\yng(2,2,2,2)}
\put(27.2,2.6){0}
\put(30.2,1){0}
\put(30.2,2){0}
\put(30.2,3){0}
\put(30.2,4){0}
\put(33,3){\yng(2,2)}
\put(32.2,3.6){0}
\put(35.2,3.1){0}
\put(35.2,4.1){0}
\put(38,3){\yng(2,2)}
\put(37.2,3.6){0}
\put(40.2,3.1){0}
\put(40.2,4.1){0}
\end{picture}
\end{center}

\begin{center}
\unitlength 10pt
\begin{picture}(3,5)
\newcommand{\kuu}{{}}
\put(-0.5,2.0){$\delta_2$}
\put(2,0){\young(\kuu\kuu1,\kuu\kuu,\kuu1,2)}
\put(1,5){\vector(0,-1){6}}
\end{picture}
\end{center}

\begin{center}
\unitlength 10pt
\begin{picture}(40,5)(2,0)
\put(18.95,1.95){\gnode}
\put(23.95,1.95){\gnode}
\put(28.95,1.95){\gnode}
\put(38.95,2.97){\gnode}
\Yboxdim{10pt}
\put(1,0){\yng(3,3,2,2,2)}
\put(0.2,1.1){0}
\put(0.2,3.6){1}
\put(3.2,0){0}
\put(3.2,1){0}
\put(3.2,2){0}
\put(4.2,3){0}
\put(4.2,4){1}
\put(7,1){\yng(3,3,2,2)}
\put(6.2,1.6){1}
\put(6.2,3.6){0}
\put(9.2,1){1}
\put(9.2,2){1}
\put(10.2,3){0}
\put(10.2,4){0}
\put(13,1){\yng(3,2,2,1)}
\put(12.2,1.1){0}
\put(12.2,2.6){1}
\put(12.2,4){1}
\put(14.2,1.1){0}
\put(15.2,2){0}
\put(15.2,3){0}
\put(16.2,4){0}
\put(18,1){\yng(2,2,2,1)}
\put(17.2,1.1){0}
\put(17.2,3){0}
\put(19.2,1.1){0}
\put(20.2,2){0}
\put(20.2,3){0}
\put(20.2,4){0}
\put(23,1){\yng(2,2,2,1)}
\put(22.2,1.1){0}
\put(22.2,3){0}
\put(24.2,1.1){0}
\put(25.2,2){0}
\put(25.2,3){0}
\put(25.2,4){0}
\put(28,1){\yng(2,2,2,1)}
\put(27.2,1.1){0}
\put(27.2,3){0}
\put(29.2,1.1){0}
\put(30.2,2){0}
\put(30.2,3){0}
\put(30.2,4){0}
\put(33,3){\yng(2,1)}
\put(32.2,3.1){0}
\put(32.2,4.1){1}
\put(34.2,3.1){0}
\put(35.2,4.1){0}
\put(38,3){\yng(2,2)}
\put(36.6,3.7){$-1$}
\put(40.2,3.1){0}
\put(40.2,4.1){0}
\end{picture}
\end{center}

\begin{center}
\unitlength 10pt
\begin{picture}(3,5)
\newcommand{\kuu}{{}}
\put(-0.5,2.0){$\delta_2$}
\put(2,0){\young(\kuu\kuu1,\kuu\kuu,\kuu1,22)}
\put(1,5){\vector(0,-1){6}}
\end{picture}
\end{center}

\begin{center}
\unitlength 10pt
\begin{picture}(40,5)(2,0)
\put(23.95,2.92){\gnode}
\put(28.95,2.92){\gnode}
\put(33.95,3.95){\gnode}
\Yboxdim{10pt}
\put(1,0){\yng(3,3,2,2,2)}
\put(0.2,1.1){0}
\put(0.2,3.6){1}
\put(3.2,0){0}
\put(3.2,1){0}
\put(3.2,2){0}
\put(4.2,3){0}
\put(4.2,4){1}
\put(7,1){\yng(3,3,2,2)}
\put(6.2,1.6){1}
\put(6.2,3.6){0}
\put(9.2,1){1}
\put(9.2,2){1}
\put(10.2,3){0}
\put(10.2,4){0}
\put(13,1){\yng(3,2,2,1)}
\put(12.2,1.1){0}
\put(12.2,2.6){0}
\put(12.2,4){0}
\put(14.2,1.1){0}
\put(15.2,2){0}
\put(15.2,3){0}
\put(16.2,4){0}
\put(18,1){\yng(2,2,1,1)}
\put(17.2,1.6){0}
\put(17.2,3.6){1}
\put(19.2,1.1){0}
\put(19.2,2){0}
\put(20.2,3){0}
\put(20.2,4){0}
\put(23,1){\yng(2,2,1,1)}
\put(22.2,1.6){0}
\put(22.2,3.6){0}
\put(24.2,1.1){0}
\put(24.2,2){0}
\put(25.2,3){0}
\put(25.2,4){0}
\put(28,1){\yng(2,2,1,1)}
\put(27.2,1.6){0}
\put(27.2,3.6){0}
\put(29.2,1.1){0}
\put(29.2,2){0}
\put(30.2,3){0}
\put(30.2,4){0}
\put(33,3){\yng(2,1)}
\put(32.2,3.1){0}
\put(32.2,4.1){0}
\put(34.2,3.1){0}
\put(35.2,4.1){0}
\put(38,3){\yng(2,1)}
\put(37.2,3.1){0}
\put(37.2,4.1){0}
\put(39.2,3.1){0}
\put(40.2,4.1){0}
\end{picture}
\end{center}

\begin{center}
\unitlength 10pt
\begin{picture}(3,6)
\newcommand{\kuu}{{}}
\put(-0.5,2.5){$\delta_2$}
\put(2,0){\young(\kuu\kuu1,\kuu\kuu,\kuu1,22,3)}
\put(1,6){\vector(0,-1){7}}
\end{picture}
\end{center}

\begin{center}
\unitlength 10pt
\begin{picture}(40,5)(2,0)
\put(28.95,3.9){\gnode}
\put(38.95,3.95){\gnode}
\Yboxdim{10pt}
\put(1,0){\yng(3,3,2,2,2)}
\put(0.2,1.1){0}
\put(0.2,3.6){1}
\put(3.2,0){0}
\put(3.2,1){0}
\put(3.2,2){0}
\put(4.2,3){0}
\put(4.2,4){1}
\put(7,1){\yng(3,3,2,2)}
\put(6.2,1.6){1}
\put(6.2,3.6){0}
\put(9.2,1){1}
\put(9.2,2){1}
\put(10.2,3){0}
\put(10.2,4){0}
\put(13,1){\yng(3,2,2,1)}
\put(12.2,1.1){0}
\put(12.2,2.6){0}
\put(12.2,4){0}
\put(14.2,1.1){0}
\put(15.2,2){0}
\put(15.2,3){0}
\put(16.2,4){0}
\put(18,1){\yng(2,2,1,1)}
\put(17.2,1.6){0}
\put(17.2,3.6){0}
\put(19.2,1.1){0}
\put(19.2,2){0}
\put(20.2,3){0}
\put(20.2,4){0}
\put(23,1){\yng(2,1,1,1)}
\put(22.2,2.1){0}
\put(22.2,4.0){1}
\put(24.2,1.1){0}
\put(24.2,2){0}
\put(24.2,3){0}
\put(25.2,4){0}
\put(28,1){\yng(2,1,1,1)}
\put(27.2,2.1){0}
\put(27.2,4.1){0}
\put(29.2,1.1){0}
\put(29.2,2){0}
\put(29.2,3){0}
\put(30.2,4){0}
\put(33,3){\yng(1,1)}
\put(32.2,3.6){0}
\put(34.2,3.1){0}
\put(34.2,4.1){0}
\put(38,3){\yng(2,1)}
\put(37.2,3.1){0}
\put(36.6,4.1){$-1$}
\put(39.2,3.1){0}
\put(40.2,4.1){0}
\end{picture}
\end{center}

\begin{center}
\unitlength 10pt
\begin{picture}(3,7)
\newcommand{\kuu}{{}}
\put(-0.5,3){$\delta_2$}
\put(2,0){\young(\kuu\kuu1,\kuu\kuu,\kuu1,22,3,4)}
\put(1,7){\vector(0,-1){8}}
\end{picture}
\end{center}

\begin{center}
\unitlength 10pt
\begin{picture}(40,5)(2,0)
\put(1.95,-0.02){\gnode}
\put(7.95,0.95){\gnode}
\put(12.95,0.96){\gnode}
\put(17.95,0.95){\gnode}
\put(22.95,0.95){\gnode}
\put(27.95,0.95){\gnode}
\put(32.95,2.96){\gnode}
\Yboxdim{10pt}
\put(1,0){\yng(3,3,2,2,2)}
\put(0.2,1.1){0}
\put(0.2,3.6){1}
\put(3.2,0){0}
\put(3.2,1){0}
\put(3.2,2){0}
\put(4.2,3){0}
\put(4.2,4){1}
\put(7,1){\yng(3,3,2,2)}
\put(6.2,1.6){1}
\put(6.2,3.6){0}
\put(9.2,1){1}
\put(9.2,2){1}
\put(10.2,3){0}
\put(10.2,4){0}
\put(13,1){\yng(3,2,2,1)}
\put(12.2,1.1){0}
\put(12.2,2.6){0}
\put(12.2,4){0}
\put(14.2,1.1){0}
\put(15.2,2){0}
\put(15.2,3){0}
\put(16.2,4){0}
\put(18,1){\yng(2,2,1,1)}
\put(17.2,1.6){0}
\put(17.2,3.6){0}
\put(19.2,1.1){0}
\put(19.2,2){0}
\put(20.2,3){0}
\put(20.2,4){0}
\put(23,1){\yng(2,1,1,1)}
\put(22.2,2.1){0}
\put(22.2,4.0){0}
\put(24.2,1.1){0}
\put(24.2,2){0}
\put(24.2,3){0}
\put(25.2,4){0}
\put(28,1){\yng(1,1,1,1)}
\put(27.2,2.6){0}
\put(29.2,1.1){0}
\put(29.2,2){0}
\put(29.2,3){0}
\put(29.2,4){0}
\put(33,3){\yng(1,1)}
\put(32.2,3.6){0}
\put(34.2,3.1){0}
\put(34.2,4.1){0}
\put(38,3){\yng(1,1)}
\put(37.2,3.6){0}
\put(39.2,3.1){0}
\put(39.2,4.1){0}
\end{picture}
\end{center}

\begin{center}
\unitlength 10pt
\begin{picture}(3,7)
\newcommand{\kuu}{{}}
\put(-0.5,3){$\delta_1$}
\put(2,0){\young(\kuu\kuu11,\kuu\kuu,\kuu1,22,3,4)}
\put(1,7){\vector(0,-1){8}}
\end{picture}
\end{center}

\begin{center}
\unitlength 10pt
\begin{picture}(40,5)(2,0)
\put(7.95,1.95){\gnode}
\put(13.95,1.96){\gnode}
\put(17.97,1.97){\gnode}
\put(22.97,1.97){\gnode}
\put(27.97,1.97){\gnode}
\put(37.95,2.96){\gnode}
\Yboxdim{10pt}
\put(1,0){\yng(3,3,2,2,1)}
\put(0.2,0.1){1}
\put(0.2,1.6){1}
\put(0.2,3.6){2}
\put(2.2,0.1){1}
\put(3.2,1){0}
\put(3.2,2){0}
\put(4.2,3){0}
\put(4.2,4){1}
\put(7,1){\yng(3,3,2,1)}
\put(6.2,1.1){0}
\put(6.2,2.1){1}
\put(6.2,3.6){0}
\put(8.2,1){0}
\put(9.2,2){1}
\put(10.2,3){0}
\put(10.2,4){0}
\put(13,2){\yng(3,2,2)}
\put(12.2,2.6){0}
\put(12.2,4){0}
\put(15.2,2){0}
\put(15.2,3){0}
\put(16.2,4){0}
\put(18,2){\yng(2,2,1)}
\put(17.2,2.1){0}
\put(17.2,3.6){0}
\put(19.2,2.1){0}
\put(20.2,3){0}
\put(20.2,4){0}
\put(23,2){\yng(2,1,1)}
\put(22.2,2.6){0}
\put(22.2,4.0){0}
\put(24.2,2){0}
\put(24.2,3){0}
\put(25.2,4){0}
\put(28,2){\yng(1,1,1)}
\put(27.2,3.1){0}
\put(29.2,2){0}
\put(29.2,3){0}
\put(29.2,4){0}
\put(33,4){\yng(1)}
\put(32.2,4.1){1}
\put(34.2,4.1){0}
\put(38,3){\yng(1,1)}
\put(36.6,3.7){$-1$}
\put(39.2,3.1){0}
\put(39.2,4.1){0}
\end{picture}
\end{center}

\begin{center}
\unitlength 10pt
\begin{picture}(3,7)
\newcommand{\kuu}{{}}
\put(-0.5,3){$\delta_1$}
\put(2,0){\young(\kuu\kuu11,\kuu\kuu2,\kuu1,22,3,4)}
\put(1,7){\vector(0,-1){8}}
\end{picture}
\end{center}

\begin{center}
\unitlength 10pt
\begin{picture}(40,5)(2,0)
\put(13.95,2.96){\gnode}
\put(18.97,2.97){\gnode}
\put(22.97,2.97){\gnode}
\put(27.97,2.97){\gnode}
\put(32.95,3.96){\gnode}
\Yboxdim{10pt}
\put(1,0){\yng(3,3,2,2,1)}
\put(0.2,0.1){1}
\put(0.2,1.6){0}
\put(0.2,3.6){1}
\put(2.2,0.1){1}
\put(3.2,1){0}
\put(3.2,2){0}
\put(4.2,3){0}
\put(4.2,4){1}
\put(7,1){\yng(3,3,1,1)}
\put(6.2,1.6){0}
\put(6.2,3.6){1}
\put(8.2,1){0}
\put(8.2,2){0}
\put(10.2,3){0}
\put(10.2,4){0}
\put(13,2){\yng(3,2,1)}
\put(12.2,2.1){0}
\put(12.2,3.1){0}
\put(12.2,4.1){0}
\put(14.2,2.1){0}
\put(15.2,3.1){0}
\put(16.2,4.1){0}
\put(18,3){\yng(2,2)}
\put(17.2,3.6){0}
\put(20.2,3){0}
\put(20.2,4){0}
\put(23,3){\yng(2,1)}
\put(22.2,3.1){0}
\put(22.2,4.0){0}
\put(24.2,3.1){0}
\put(25.2,4){0}
\put(28,3){\yng(1,1)}
\put(27.2,3.6){0}
\put(29.2,3){0}
\put(29.2,4){0}
\put(33,4){\yng(1)}
\put(32.2,4.1){0}
\put(34.2,4.1){0}
\put(38,4){\yng(1)}
\put(37.2,4.1){0}
\put(39.2,4.1){0}
\end{picture}
\end{center}

\begin{center}
\unitlength 10pt
\begin{picture}(3,7)
\newcommand{\kuu}{{}}
\put(-0.5,3){$\delta_1$}
\put(2,0){\young(\kuu\kuu11,\kuu\kuu2,\kuu13,22,3,4)}
\put(1,7){\vector(0,-1){8}}
\end{picture}
\end{center}

\begin{center}
\unitlength 10pt
\begin{picture}(40,5)(2,0)
\put(18.97,3.97){\gnode}
\put(23.97,3.97){\gnode}
\put(27.97,3.97){\gnode}
\put(37.95,3.96){\gnode}
\Yboxdim{10pt}
\put(1,0){\yng(3,3,2,2,1)}
\put(0.2,0.1){1}
\put(0.2,1.6){0}
\put(0.2,3.6){1}
\put(2.2,0.1){1}
\put(3.2,1){0}
\put(3.2,2){0}
\put(4.2,3){0}
\put(4.2,4){1}
\put(7,1){\yng(3,3,1,1)}
\put(6.2,1.6){0}
\put(6.2,3.6){0}
\put(8.2,1){0}
\put(8.2,2){0}
\put(10.2,3){0}
\put(10.2,4){0}
\put(13,2){\yng(3,1,1)}
\put(12.2,2.6){0}
\put(12.2,4.1){1}
\put(14.2,2.1){0}
\put(14.2,3.1){0}
\put(16.2,4.1){0}
\put(18,3){\yng(2,1)}
\put(17.2,3.1){0}
\put(17.2,4.1){0}
\put(19.2,3){0}
\put(20.2,4){0}
\put(23,4){\yng(2)}
\put(22.2,4.1){0}
\put(25.2,4.1){0}
\put(28,4){\yng(1)}
\put(27.2,4.1){0}
\put(29.2,4.1){0}
\put(33.4,4){$\emptyset$}
\put(38,4){\yng(1)}
\put(36.6,4.1){$-1$}
\put(39.2,4.1){0}
\end{picture}
\end{center}

\begin{center}
\unitlength 10pt
\begin{picture}(3,7)
\newcommand{\kuu}{{}}
\put(-0.5,3){$\delta_1$}
\put(2,0){\young(\kuu\kuu11,\kuu\kuu2,\kuu13,224,3,4)}
\put(1,7){\vector(0,-1){8}}
\end{picture}
\end{center}

\begin{center}
\unitlength 10pt
\begin{picture}(40,5)(2,0)
\Yboxdim{10pt}
\put(1,0){\yng(3,3,2,2,1)}
\put(0.2,0.1){1}
\put(0.2,1.6){0}
\put(0.2,3.6){1}
\put(2.2,0.1){1}
\put(3.2,1){0}
\put(3.2,2){0}
\put(4.2,3){0}
\put(4.2,4){1}
\put(7,1){\yng(3,3,1,1)}
\put(6.2,1.6){0}
\put(6.2,3.6){0}
\put(8.2,1){0}
\put(8.2,2){0}
\put(10.2,3){0}
\put(10.2,4){0}
\put(13,2){\yng(3,1,1)}
\put(12.2,2.6){0}
\put(12.2,4.1){0}
\put(14.2,2.1){0}
\put(14.2,3.1){0}
\put(16.2,4.1){0}
\put(18,3){\yng(1,1)}
\put(17.2,3.6){0}
\put(19.2,3){0}
\put(19.2,4){0}
\put(23,4){\yng(1)}
\put(22.2,4.1){0}
\put(24.2,4.1){0}
\put(28.4,4){$\emptyset$}
\put(33.4,4){$\emptyset$}
\put(38.4,4){$\emptyset$}
\end{picture}
\end{center}
The final rigged configuration and $T$ of the above diagrams
give the image of $\Psi$.
Under the bijection \cite{KR} the final rigged configuration
corresponds to the following element:
\[
\Yvcentermath1
p'=
\young(111)\otimes\young(222)\otimes
\young(133)\otimes\young(44)\otimes
\young(35)\otimes\young(4)\otimes\young(6).
\]
\end{example}

\begin{remark}
Let $p$ and $p'$ as in Example \ref{ex:Psi}
and again consider them as elements of $D^{(1)}_8$.
If we apply the involution $\sigma$ at Section 5.3 of \cite{LOS},
we have
\[
\Yvcentermath1
\Yboxdim 12pt
\newcommand{\beight}{\bar{8}}
\newcommand{\bseven}{\bar{7}}
\newcommand{\bsix}{\bar{6}}
\sigma(p)=
\young(\beight\beight\beight)\otimes\young(88\bseven)\otimes
\young(6\beight\bsix)\otimes\young(\bseven\bsix)\otimes
\young(6\bsix)\otimes\young(7)\otimes\young(\bseven).
\]
Then $p'$ coincides with the $I_0$-highest element
corresponding to $\sigma (p)$.
We expect that the same relation holds for arbitrary
image of $\Psi$.
\end{remark}

Now we are going to
give the description of the algorithm $\tilde{\Psi}$
which will be shown to be the inverse of $\Psi$.
\begin{definition}\label{def:delta-1}
The map $\tilde{\delta}_k$
\[
\tilde{\delta}_k:(\nu^\bullet,J^\bullet)
\longmapsto (\nu'^\bullet,J'^\bullet)
\]
is defined by the following algorithm.
Here the integer $k$ should satisfy $k\leq a^\diamondsuit$.
\begin{enumerate}
\item[(i)]
Starting with $\nu^{(k)}$, choose rows of $\nu^{(a)}$
($a\leq a^\diamondsuit$) recursively as follows.
To initialize the process,
let us tentatively assume that we have chosen an
infinitely long row of $\nu^{(k-1)}$.
Suppose that we have chosen a row of $\nu^{(a-1)}$.
Find a longest singular row of $\nu^{(a)}$ whose length
does not exceed the length of the chosen row of $\nu^{(a-1)}$.
If there is no such row, suppose that we have chosen
a length 0 row of $\nu^{(a)}$ and continue.
Otherwise choose one such singular row and continue.
\item[(ii)]
Suppose that we have chosen the length $l$ row of $\nu^{(a^\diamondsuit)}$.
To finish the process, do one of the following:
\begin{enumerate}
\item If $\diamondsuit =\hdom$ or $\sdom$,
choose one of the length $l$ rows
of $\nu^{(n)}$.
\item The case $\diamondsuit =\vdom$.
If $|\nu^{(n-1)}|=|\nu^{(n)}|$, choose
one of the length $l$ rows of $\nu^{(n)}$.
On the other hand, if $|\nu^{(n-1)}|<|\nu^{(n)}|$,
choose one of the length $l$
rows of $\nu^{(n-1)}$.
\end{enumerate}
\item[(iii)]
$\nu'^\bullet$ is obtained by adding one box to each chosen row
in Step (i) and (ii).
If the length of the chosen row is 0,
create a new row at the bottom of the corresponding
partition $\nu^{(a)}$.
\item[(iv)]
The new riggings $J'^\bullet$ are defined as follows.
Take all entries of $J'^{(a)}$ to be 0 
for $a>a^\diamondsuit$.
The remaining parts are defined as follows.
For the rows that are not changed in Step (iii),
take the same riggings as before.
Otherwise set the new riggings equal to the corresponding
vacancy numbers computed by using $\nu'^\bullet$.
\end{enumerate}
\end{definition}

\begin{definition}\label{def:Psi-1}
The map $\tilde{\Psi}$
\[
\tilde{\Psi}:
\{(\nu^\bullet,J^\bullet),T\}
\longmapsto
(\nu'^\bullet,J'^\bullet)
\]
is defined as follows.
\begin{enumerate}
\item[(i)]
Let $h_1$ be the largest integer contained in $T$.
For $h_1$ do the following procedure.
Fix the rightmost occurrence of $i$ in $T$ for each $1\le i\le h_1$.
Call these fixed $h_1$ integers of $T$ the first group.
Remove all members of the first group from $T$
and do the same procedure for the new $T$.
Call the integers that are fixed this time the second group.
Repeat the same procedure recursively until all integers
of $T$ are grouped.
Let the total number of groups be $l$,
the cardinality of the $i$-th group be $h_i$ and
the position of the letter $j$ contained in the $i$-th group
be the $k_{i,j}$-th row (counting from the top of $T$).
\item[(ii)]
The output of $\tilde{\Psi}$ is defined as follows:
\[
(\nu'^\bullet,J'^\bullet)=
\tilde{\delta}_{k_{l,1}}\cdots\cdots
\tilde{\delta}_{k_{2,1}}
\tilde{\delta}_{k_{2,2}}\cdots
\tilde{\delta}_{k_{2,h_2}}
\tilde{\delta}_{k_{1,1}}
\tilde{\delta}_{k_{1,2}}\cdots
\tilde{\delta}_{k_{1,h_1}}
(\nu^\bullet,J^\bullet).
\]
\end{enumerate}
\end{definition}

\begin{example}
As for an example of $\tilde{\Psi}$,
one should read Example \ref{ex:Psi}
in the reverse order.
More precisely, reverse all arrows and apply
$\tilde{\delta}_4$, $\tilde{\delta}_3$, $\tilde{\delta}_2$,
$\tilde{\delta}_1$, $\tilde{\delta}_6$, $\tilde{\delta}_5$,
$\tilde{\delta}_4$, $\tilde{\delta}_1$, $\tilde{\delta}_4$,
$\tilde{\delta}_3$ in this order.
\end{example}

\subsection{Main statements}
The crux of the combinatorics is contained
in the following two theorems on the well-definedness
of both maps $\Psi$ and $\tilde{\Psi}$, which we will prove later.

\begin{theorem}\label{Psi_weldef}
Assume that $(\nu^\bullet,J^\bullet)\in
{\rm RC}^\diamondsuit$.
Suppose that the rank $n$ satisfies
$a^\diamondsuit\geq \ell ({\rm wt}(\nu^\bullet,J^\bullet))
+\ell (\nu^{(a^\diamondsuit)})$.
Then the map $\Psi$
\[
\Psi : (\nu^\bullet,J^\bullet)\longmapsto
\{(\nu'^\bullet,J'^\bullet),T\}
\]
is well-defined.
More precisely, $(\nu'^\bullet,J'^\bullet)
\in {\rm RC}^\varnothing$ and the LR tableau
$T\in LR^\eta_{\lambda\mu}$
satisfy the following properties:
\[
\lambda=\mathrm{wt}(\nu^\bullet,J^\bullet),\qquad
\mu=\nu^{(a^\diamondsuit)},\qquad
\eta=\mathrm{wt}(\nu'^\bullet,J'^\bullet).
\]
\end{theorem}

\begin{theorem}\label{Psi-1_weldef}
Assume that $(\nu^\bullet,J^\bullet)
\in {\rm RC}^\varnothing$ and $T$ is the
LR tableau that satisfy
the following three properties:
$T\in LR^\eta_{\lambda\mu}$ where
$\lambda,\mu\in\mathcal{P}^\diamondsuit$
and $\eta=\mathrm{wt}(\nu^\bullet,J^\bullet)$.
Then the map $\tilde{\Psi}$;
\[
\tilde{\Psi}:
\{(\nu^\bullet,J^\bullet),T\}
\longmapsto
(\nu'^\bullet,J'^\bullet),
\]
is well-defined.
More precisely, we have
$(\nu'^\bullet,J'^\bullet)\in {\rm RC}^\diamondsuit$,
$\mathrm{wt}(\nu'^\bullet,J'^\bullet)=\lambda$ and
$\nu'^{(a^\diamondsuit)}=\mu$.
\end{theorem}

By construction, $\delta$ and $\tilde{\delta}$ are
mutually inverse procedures.
Therefore the above theorems imply the following
main theorem.

\begin{theorem}\label{th:main}
Assume that $(\nu^\bullet,J^\bullet)\in
{\rm RC}^\diamondsuit$.
Suppose that the rank $n$ satisfies
$a^\diamondsuit\geq \ell ({\rm wt}(\nu^\bullet,J^\bullet))
+\ell (\nu^{(a^\diamondsuit)})$.
Then $\Psi$ gives a bijection between ${\rm RC}^\diamondsuit$
and the disjoint union of product sets of ${\rm RC}^\varnothing$
and the LR tableaux as follows:
\begin{align*}
\Psi:\;&
(\nu^\bullet,J^\bullet)\longmapsto\{(\nu'^\bullet,J'^\bullet),T\},\\
&(\nu^\bullet,J^\bullet)\in\mathrm{RC}^\diamondsuit (\lambda,\mathbf{L}),\quad
\{(\nu'^\bullet,J'^\bullet),T\}\in
\mathrm{RC}^\varnothing(\eta,\mathbf{L})\times LR^\eta_{\lambda\mu},
\end{align*}
where $\lambda$, $\mu$, $\eta$
satisfy
\[
\lambda=\mathrm{wt}(\nu^\bullet,J^\bullet),\qquad
\mu=\nu^{(a^\diamondsuit)},\qquad
\eta=\mathrm{wt}(\nu'^\bullet,J'^\bullet).
\]
That is, for fixed $\lambda$ and $\mathbf{L}$, $\Psi$ defines a bijection
\[
\mathrm{RC}^\diamondsuit(\lambda,\mathbf{L}) \longrightarrow
\bigsqcup_{\mu\in\mathcal{P}^\dia_{|\El|-|\la|},\eta\in\mathcal{P}^\cells_{|\El|}}
\mathrm{RC}^\varnothing(\eta,\mathbf{L}) \times LR^\eta_{\lambda\mu}.
\]
The inverse procedure is given by $\Psi^{-1}=\tilde{\Psi}$.
\end{theorem}

Let us examine the restrictions on the rank $n$
for $\Psi$ and $\tilde{\Psi}$.
For simplicity, consider $\diamondsuit =\vdom$.
Let $N_{out}$ (resp. $N_{in}$) be the length of
the outer (resp. inner) shape of $T$ and $|T|$ be
the number of (either filled or not filled) boxes in $T$.
Then, by construction, $\tilde{\Psi}$ is defined
if the rank $n$ satisfies $N_{out}+2\leq n$.
Let us consider $\Psi$.
Following Theorem \ref{th:main}, we see that
the weight of the LR tableau
$T$ coincide with the partition $\nu^{(a^\diamondsuit)}$
of ${\rm RC}^\diamondsuit$ and the inner shape
of $T$ coincides with ${\rm wt}(\nu^\bullet,J^\bullet)$.
Since $T$ has column strict property, we have
$N_{out}-N_{in}\leq \ell(\nu^{(a^\diamondsuit)})$.
On the other hand,
the condition in Theorem \ref{Psi_weldef} can be rephrased as
$N_{in}+\ell(\nu^{(a^\diamondsuit)})+2\leq n$.
To summarize the condition on the rank $n$ for $\Psi$ case is
more restrictive and requires the rank $n$
to be larger than that for $\tilde{\Psi}$.
We remark that in the present definition of $\Psi$, the partition
$\nu^{(a^\diamondsuit)}$ is always removed entirely.
For this point one can modify the definition of $\Psi$ slightly
so that the required condition on the rank $n$ decreases
by 1 compared with the original $\Psi$.

In order to get a feeling about the condition on the rank,
let us recall the bijection between ${\rm RC}^\varnothing$
and tensor product of type $A^{(1)}_n$ crystals in \cite{KSS}.
Then if ${\rm RC}^\varnothing$ corresponds to an element of
$\bigotimes_i B^{r_i,s_i}$, we have $|T|=\sum_i r_is_i$.
Thus the quantity $\ell(\nu^{(a^\diamondsuit)})$
is always smaller than or equal to $\sum_i r_is_i$
(usually, much smaller than $\sum_i r_is_i$).
Moreover, recall the bijection between ${\rm RC}^\vdom$
and tensor products of type $D^{(1)}_n$ crystals in \cite{OSS,SS}.
Let the number of barred letters (i.e., $\bar{1},\bar{2},\ldots$)
of the corresponding tensor product of crystals be $M$.
Recall that under the present condition on the rank
there is no $\bar{n}$.
Then the total number of letters filled in $T$ is equal to $2M$.

\subsection{Proof of Theorem \ref{Psi_weldef}}
Let us fix notations that will be used in this subsection.
We will refer to Step (i), etc., of Definition \ref{def:delta}
simply as Step (i), etc.
We denote by $l^{(a)}$ the length of a
row that is removed from $\nu^{(a)}$ by $\delta_l$.
In other words, $l^{(a)}$ is the position of the box
to be removed by $\delta_l$.
Then we have $l^{(a-1)}\geq l^{(a)}\geq l^{(a+1)}$.
We denote by $(\nu'^\bullet,J'^\bullet)$
the image of $\delta_l$ and $p_l'^{(a)}$ the vacancy
number with respect to $\nu'^\bullet$.
If we further apply $\delta_l$ on $(\nu'^\bullet,J'^\bullet)$,
we use $l'^{(a)}$ to express the position of the
box to be removed by the second $\delta_l$.
We use the symbol $\Delta$ to express the differences
of the quantities for $(\nu'^\bullet,J'^\bullet)$
minus $(\nu^\bullet,J^\bullet)$.
For example, $\Delta p^{(a)}_l=p'^{(a)}_l-p^{(a)}_l$.
Let the height of the $i$-th column (counting from
left) of $\nu^{(a^\diamondsuit)}$ be $h_i$.
Note that the procedure $\Psi$ does not change the quantum space.
Therefore if the arguments concern only the differences of
the vacancy numbers, we can neglect the effect of
the quantum space on the vacancy numbers.

During this subsection, we assume that the rank $n$
satisfies the condition
\begin{equation}\label{eq:condition_on_rank}
\ell({\rm wt}
(\nu^\bullet,J^\bullet))+\ell(\nu^{(a^\diamondsuit)})
\leq a^\diamondsuit.
\end{equation}
We start with the following lemma, that makes the statement of Proposition \ref{prop:stability} (1)
more precise.
\begin{lemma} \label{lem:large_rank_RC}
Set $l=\nu^{(a^\diamondsuit)}_1$ and
$a_{l+1}=\ell(\wt(\nu^\bullet,J^\bullet))+1$. Then there exists an integer sequence
$a_{l+1}\leq a_l\leq a_{l-1}\leq\cdots\leq a_1$ that satisfies
\begin{align*}
&a_i-a_{i+1}\leq m^{(a^\diamondsuit)}_i,\\
&m^{(a_{i+1})}_i<m^{(a_{i+1}+1)}_i<\cdots <m^{(a_{i})}_i
=\cdots =m^{(a^\diamondsuit)}_i
\end{align*}
for any $i$ such that $l\geq i\geq 1$.
\end{lemma}
\begin{proof}
We give the proof for $\diamondsuit =\vdom$.
Proofs for the other cases are similar.
By Lemma \ref{lem:width} and
$|\nu^{(n-2)}|=2|\nu^{(n)}|$, we see that $p^{(n)}_l=0$.
On the other hand, from Corollary \ref{cor:positivity},
we have $p^{(n)}_{l-1}\geq 0$.
Combining these two relations, we have
$p^{(n)}_l-p^{(n)}_{l-1}=(Q^{(n-2)}_l-Q^{(n-2)}_{l-1})
-2(Q^{(n)}_l-Q^{(n)}_{l-1})=m^{(n-2)}_l-2m^{(n)}_l\leq 0$,
hence $m^{(n-2)}_l\leq 2m^{(n)}_l$.
Similarly, we have $m^{(n-2)}_l\leq 2m^{(n-1)}_l$.
Next, from $p^{(n-2)}_l=0$ and $p^{(n-2)}_{l-1}\geq 0$,
we obtain $m_l^{(n-3)}-2m_l^{(n-2)}+m_l^{(n-1)}+m_l^{(n)}\leq 0$.
Combining this with the previously obtained inequalities
$m^{(n-2)}_l/2-m^{(n)}_l\leq 0$ and
$m^{(n-2)}_l/2-m^{(n-1)}_l\leq 0$, we deduce
$m_l^{(n-3)}\leq m_l^{(n-2)}$.
We can recursively do the same arguments and obtain
$m^{(a-1)}_l\leq m^{(a)}_l$ for $k<a\leq n-2$.
Moreover, from the same arguments, we see that once
$m^{(a'-1)}_l<m^{(a')}_l$ happens then
$m^{(a-1)}_l<m^{(a)}_l$ holds for all $k<a\leq a'$.
Define $a_l$ to be the maximal such $a'$.
If there is no such $a'$, set $a_l=k+1$.
Then if $m^{(n-2)}_l<a_l-k$ we have $m^{(k+1)}_l\leq 0$,
which contradicts with Lemma \ref{lem:width}.
Therefore we conclude that $a_l-k\leq m^{(n-2)}_l$.
In order to go further, we observe the following.
By definition of $a_l$, we have
$m^{(a_l)}_l=m^{(a_l+1)}_l=\cdots =m^{(n-2)}_l=
m^{(n-1)}_l/2=m^{(n)}_l/2$.
{}From this we can deduce $p^{(a_l+1)}_{l-1}=p^{(a_l+2)}_{l-1}=
\cdots =p^{(n)}_{l-1}=0$.
Then by the same arguments that are given in the previous case, we have
$m^{(a-1)}_{l-1}\leq m^{(a)}_{l-1}$ for $a_l<a\leq n-2$.
Again, in decreasing $a$, once the inequality is strict
so are the rest of inequalities.
Suppose that $a_{l-1}$ is the maximal index where such a strict
inequality appears.
This time, in order to satisfy the condition
$m^{(a_l)}_{l-1}\geq 0$,
$a_{l-1}$ has to satisfy $a_{l-1}-a_l\leq m^{(n-2)}_{l-1}$.
Starting from $a_{l-1}$, we can recursively continue the
similar arguments and finish the proof.
\end{proof}
\begin{prop}\label{prop:delta_weldef}
$\delta_l$ is well defined
if $\diamondsuit =\hdom$ or $\sdom$,
and $\delta_l^2$ is well defined
if $\diamondsuit =\vdom$.
\end{prop}
\begin{proof}
In order to check the well-definedness of $\delta_l$
and $\delta_l^2$, we need to check the positivity of
the vacancy numbers and the inequality that
the riggings are less than or equal to
the corresponding vacancy numbers.
To do this, we have to distinguish the following three cases.

Let us remark that, from the condition
(\ref{eq:condition_on_rank}) on the rank $n$,
we can use Proposition \ref{prop:stability}
so that we can assume the specific nature of the rigged configurations
stated in that proposition.
\bigskip\\
\underline{Case 1.}
Let us check the well-definedness for $\nu^{(a)}$
($a\geq a^\diamondsuit$).
First check the cases $\diamondsuit =\hdom$ or $\sdom$.
{}From Proposition \ref{prop:stability},
we see that $\nu^{(n-1)}=\nu^{(n)}$ and $J^{(n)}_i=0$ for all $i$.
In Step (i), $\delta_l$ removes
one box from each $l$-th column of $\nu^{(n-1)}$ and $\nu^{(n)}$.
Thus we have $\nu'^{(n-1)}=\nu'^{(n)}$
so that the vacancy numbers for $\nu'^{(n)}$ are all 0.
This shows the well-definedness for $\nu^{(n)}$.
Let us consider $\nu^{(n-1)}$.
In Step (iii), $\delta_l$ may remove a
row of $\nu^{(n-2)}$ that is longer than or equal to $l$.
We see that it will contribute to $\Delta Q^{(n-2)}_l$ as 0 or $-1$.
Then the vacancy numbers $p^{(n-1)}_{j}$ for $j\geq l$
will not decrease because
$\Delta p^{(n-1)}_{j}=
\Delta Q^{(n-2)}_{j}-2\Delta Q^{(n-1)}_{j}+\Delta Q^{(n)}_{j}=
\Delta Q^{(n-2)}_{j}-2(-1)+(-1)=\Delta Q^{(n-2)}_{j}+1\geq 0$.
On the other hand, if $j<l$ then $\Delta p^{(n-1)}_{j}=0$.
This implies the well-definedness for $\nu^{(n-2)}$.

Let us consider the case $\diamondsuit =\vdom$.
Denote by $\Delta'$ the differences of the quantities
after $\delta_l^2$ minus those before $\delta_l^2$.
Then we observe that $\Delta' Q^{(n)}_l=\Delta' Q^{(n-1)}_l=-1$,
$\Delta' Q^{(n-2)}_l=-2$ and $-2\leq\Delta' Q^{(n-3)}_l\leq 0$.
Then we can use a similar discussion as in the previous case.
For example, for $j\geq l$, $\Delta'p^{(n-2)}_{j}
=\Delta'Q^{(n-3)}_{j}-2\Delta'Q^{(n-2)}_{j}
+\Delta'Q^{(n-1)}_{j}+\Delta'Q^{(n)}_{j}=
\Delta'Q^{(n-3)}_{j}-2(-2)+(-1)+(-1)=
\Delta'Q^{(n-3)}_{j}+2\geq 0$
which implies the well-definedness for $\nu^{(n-2)}$.

We remark that once the well-definedness for
$a\geq a^\diamondsuit$ is established, the rest of the
proof does not depend on the choice of $\diamondsuit$.
Especially, it is enough to consider $\delta_l$
and does not need to consider $\delta_l^2$.\bigskip\\
\underline{Case 2.}
Next let us check the well-definedness for rows of $\nu^{(a)}$
($a<a^\diamondsuit$) that are not removed by $\delta_l$.
Recall that $l^{(a-1)}\geq l^{(a)}\geq l^{(a+1)}$.
Then by using the similar arguments of the previous case
we conclude that $\Delta p^{(a)}_{j}\geq 0$ if $j\geq l^{(a)}$
and $\Delta p^{(a)}_{j}=0$ if $l^{(a+1)}>j$,
which imply the well-definedness in both cases.
Now consider the case $l^{(a)}>j\geq l^{(a+1)}$.
Let the rigging corresponding to such a length $j$ row
of $\nu^{(a)}$ be $J(\geq 0)$.
Since $l^{(a)}>j$ the row will not be removed by $\delta_l$
so that the rigging $J$ will not change.
In this case we have $\Delta p^{(a)}_{j}=-1$.
Remind that this row is not removed
although it satisfies $j\geq l^{(a+1)}$.
Thus the row is not singular
before the application of $\delta_l$
so that it satisfies $p^{(a)}_{j}-J\geq 1$.
Combining this with $\Delta p^{(a)}_{j}=-1$ we have
$p_j'^{(a)}\geq J\geq 0$.
This shows that the rigging does not exceeds
the corresponding vacancy number and also that the positivity
of the vacancy number.
\bigskip\\
\underline{Case 3.}
Consider the remaining case, i.e., rows of $\nu^{(a)}$
($a<a^\diamondsuit$) that are removed by $\delta_l$.
The point is that $\delta_l$ changes the length $l^{(a)}$ row
into $l^{(a)}-1$ row so that we have to ensure the positivity
$p'^{(a)}_{l^{(a)}-1}\geq 0$.
We do not need to consider the rigging separately
since the new rigging is taken equal to $p'^{(a)}_{l^{(a)}-1}$.
If $p^{(a)}_{l^{(a)}-1}>0$, we can use the similar arguments
of the previous case to ensure the positivity.
Therefore we assume $p^{(a)}_{l^{(a)}-1}=0$
in the rest of the proof.
If $l^{(a)}>l^{(a+1)}$ then
$p'^{(a)}_{l^{(a)}-1}=-1$ which violates the positivity.
Thus we ought to show $l^{(a)}=l^{(a+1)}$.
In order to prove this, let us show that assuming
$l^{(a)}>l^{(a+1)}$ leads to a contradiction.
To begin with, let us check that we can use Lemma
\ref{lem:convex} with $l=l^{(a)}-1$.
If there is a length $l^{(a)}-1$ row of $\nu^{(a)}$,
then it is singular (since the assumption $p^{(a)}_{l^{(a)}-1}=0$)
and it is longer than or equal to the length $l^{(a+1)}$
row of $\nu^{(a+1)}$.
Thus $\delta_l$ can remove the length $l^{(a)}-1$ row instead of
the length $l^{(a)}$ row, which is a contradiction.
Then Lemma \ref{lem:convex} combined with $p^{(a)}_{l^{(a)}-1}=0$
gives $p^{(a)}_{l^{(a)}}=p^{(a)}_{l^{(a)}-2}=0$
(recall that by definition we have $p^{(a)}_{l^{(a)}}\geq 0$
and by Corollary \ref{cor:positivity} we have
$p^{(a)}_{l^{(a)}-2}\geq 0$).
{}From $p^{(a)}_{l^{(a)}-2}=0$ we can recursively do
the same arguments and obtain the following result.
Let $j$ be the length of the longest rows
among the rows of $\nu^{(a)}$
that are strictly shorter than $l^{(a)}$.
If there is no such row, set $j=0$.
Then we have $p^{(a)}_{l^{(a)}}=p^{(a)}_{l^{(a)}-1}=\cdots
=p^{(a)}_j=0$.
In particular the length $j$ rows of $\nu^{(a)}$ are singular.
Thus the assumption $l^{(a)}>j$ forces to $l^{(a+1)}>j$
in order to avoid the removal of length $j$ row by $\delta_l$.
Then on the interval $j\leq i\leq l^{(a)}$
the following properties hold as functions of $i$:
\begin{center}
\begin{tabular}{lcl}
$Q^{(a-1)}_i$&:&upper convex function (including linear function case),\\
$Q^{(a)}_i$&:& linear function,\\
$Q^{(a+1)}_i$&:&
strictly upper convex function due to the existence of
the length\\
&& $l^{(a+1)}$ row that satisfy
$l^{(a)}>l^{(a+1)}>j$.
\end{tabular}
\end{center}
Here we have used the terminology ``strictly upper convex"
in order to express the situation that the inequality
in Lemma \ref{lem:convex} is strict.
Summing up all three contributions, we conclude that
the function $p^{(a)}_i$ is a strictly upper convex
function with respect to $i$ on this interval.
However, as we have seen, $p^{(a)}_i=0$ for
$j\leq\forall i\leq l^{(a)}$.
This is a contradiction.
Hence we have $l^{(a)}=l^{(a+1)}$ which ensures
the well-definedness of $\delta_l$ for Case 3.
\end{proof}

\begin{prop}\label{prop:Psi_stabilize}
Let us denote the image of the map $\Psi$ as
\[
\Psi :(\nu^\bullet ,J^\bullet)\longmapsto
\{ (\nu'^\bullet ,J'^\bullet),T\}.
\]
Then $\Psi$ is well-defined on $(\nu^\bullet ,J^\bullet)$ and
the image $\{ (\nu'^\bullet ,J'^\bullet),T\}$
does not depend on $n$.
Moreover, if the inequality of the condition
(\ref{eq:condition_on_rank}) on the rank $n$ is strict,
then all the boxes of $\nu^{(a^\diamondsuit -1)}$ and
$\nu^{(a^\diamondsuit)}$ are removed in exactly the same
way during the whole procedure of $\Psi$.
\end{prop}
\begin{proof}
Suppose that the rank $n$ satisfies the strict inequality
$a^\diamondsuit>k+\ell (\nu^{(a^\diamondsuit)})$.
We first show the latter statement.
Then, combining $\nu^{(k+\ell(\nu^{(a^\diamondsuit)}))}=
\nu^{(k+\ell(\nu^{(a^\diamondsuit)})+1)}=\cdots
=\nu^{(a^\diamondsuit)}$
from Proposition \ref{prop:stability} (1),
the well-definedness of $\Psi$ and
the independence of the image on the rank $n$ follow.
During the proof, let $m_j^{(a)}$ be the multiplicity
corresponding to the initial $(\nu^\bullet,J^\bullet)$,
let $l=\nu^{(a^\diamondsuit)}_1$ and let $h_i$
be the height of the $i$-th column (counting from left)
of $\nu^{(a^\diamondsuit)}$.

The first step of $\Psi$ is the application of $\delta_l^{h_l}$.
Let us show that $\delta_l^{h_l}$ remove the $l$-th columns
of $\nu^{(a^\diamondsuit -1)}$ and $\nu^{(a^\diamondsuit)}$.
{}From this, we can deduce that
the positions of the removed boxes
are the same for these two partitions
since the height of two columns are both $h_l$.
{}From Lemma \ref{lem:width} we see that the longest rows of
$\nu^{(k+1)},\ldots,\nu^{(a^\diamondsuit)}$ are
length $l$ singular rows.
Moreover, from Lemma \ref{lem:large_rank_RC},
one of the followings is true; $m^{(k+1)}_l=m^{(k+2)}_l=\cdots=
m^{(a^\diamondsuit)}_l$ or $m^{(k+1)}_l<m^{(k+2)}_l<\cdots <
m^{(a_l)}_l=\cdots =m^{(a^\diamondsuit)}_l$
for some $k+1<a_l<a^\diamondsuit$.
Then we can analyze the consequences of $\delta_l^{h_l}$ as follows.
To begin with, $\delta_l^{h_l}$ can remove only length
$l$ rows of $\nu^{(k+1)},\ldots,\nu^{(a^\diamondsuit)}$
and, in particular, the first $\delta_l$ will remove
one box from each partition
$\nu^{(k+1)},\ldots,\nu^{(a^\diamondsuit)}$.
To analyze the effects of $\delta_l$ after the first one, we have to
clarify the changes of the vacancy numbers caused by $\delta_l$.
Since the removed rows have the same length,
we can do the analysis in the following way.
Let $j>k+1$.
\begin{enumerate}
\item[(a)] If length $l$ rows of $\nu^{(j-1)},\nu^{(j)},\nu^{(j+1)}$
are removed, then the vacancy number $p^{(j)}_l$
will not change.
Recall that initially we have $p^{(j)}_l=0$
for $k<j\leq a^\diamondsuit$.
Therefore we see that when we remove length $l$ rows of
$\nu^{(j-1)},\nu^{(j)},\nu^{(j+1)}$,
the initial length $l$ singular rows of $\nu^{(j)}$ remain as
singular rows of the same length if they are not removed.
\item[(b)] If length $l$ rows of $\nu^{(j)},\nu^{(j+1)}$ are
removed but $\nu^{(j-1)}$ is not removed by the $i$-th $\delta_l$, then
the vacancy number $p_l^{(j)}$ will increase by 1.
Thus if we further apply $\delta_l$,
$\nu^{(j)}$ will not be removed
since there is no singular rows of $\nu^{(j)}$
whose length are greater than or equal to $l$.
On the other hand, $\nu^{(j+1)}$ will be removed.
To show this, we have to distinguish two cases.
If initially $m^{(j)}_l=m^{(a^\diamondsuit)}_l=h_l$,
$\delta_l^{i}$ ($i<h_l$) will not exhaust the
$l$-th column of $\nu^{(j+1)}$.
The remaining length $l$ rows are still singular
as in the above case (a).
On the contrary, if initially
$m^{(j)}_l<m^{(a^\diamondsuit)}_l=h_l$,
we have moreover $m^{(j)}_l<m^{(j+1)}_l$.
Therefore, again we have remaining length $l$ singular rows of
$\nu^{(j+1)}$ which can be removed.
\end{enumerate}
To summarize we obtain the following conclusion.
When we apply $\delta_l^{h_l}$, the $i$-th $\delta_l$
removes length $l$ rows of
$\nu^{(a)},\ldots,\nu^{(a^\diamondsuit)}$ for some $a\leq k+i$.
Now consider the case $i=h_l$.
Since $a^\diamondsuit -k>\ell (\nu^{(a^\diamondsuit)})\geq h_l$,
we have $k+h_l<a^\diamondsuit$.
In particular we see that $\delta_l^{h_l}$ will remove
both $l$-th columns of
$\nu^{(a^\diamondsuit -1)}$ and $\nu^{(a^\diamondsuit)}$,
which is the desired fact.

Next let us consider $\delta_{l-1}^{h_{l-1}}$.
To begin with, we see the following facts
concerning with the interaction with the effect
of $\delta_l^{h_l}$.
Suppose that some length $l$ rows of
$\nu^{(k+1)},\ldots,\nu^{(a^\diamondsuit)}$
are removed by $\delta_l^{h_l}$ and set to be singular.
Then all such rows remain as length $l-1$ singular rows
after applications of $\delta_l^{h_l}$.
This follows from the fact that $\delta_l^{h_l}$
remove length $l$ rows of
$\nu^{(k+1)},\ldots,\nu^{(a^\diamondsuit)}$
so that such removal does not affect the vacancy numbers $p^{(j)}_{l-1}$
for $k+1\leq j\leq a^\diamondsuit$.
Therefore we see that the $i$-th $\delta_{l-1}$
($i\leq h_l$) will remove the same rows
that have been removed by the $i$-th $\delta_l$.
In other words, the $i$-th $\delta_{l-1}$ ($i\leq h_l$) removes length $l-1$ rows of
$\nu^{(a)},\ldots,\nu^{(a^\diamondsuit)}$ for some $a\leq k+i$.

Let us consider the remaining $\delta_{l-1}^{h_{l-1}-h_l}$.
Assume that $h_{l-1}>h_l$.
Then we have $m^{(a^\diamondsuit)}_{l-1}>0$.
As before, let $a_{l}(>k+1)$ be the maximal index such that
$m^{(a_{l}-1)}_l<m^{(a_{l})}_l$ happens.
If there is no such $a_{l}$, set $a_{l}=k+1$.
Then we observe the following two facts.
(1) From Lemma \ref{lem:large_rank_RC},
we have one of the followings;
$m^{(a_{l})}_{l-1}=m^{(a_{l}+1)}_{l-1}=\cdots=
m^{(a^\diamondsuit)}_{l-1}$ or $m^{(a_{l})}_{l-1}<m^{(a_{l}+1)}_{l-1}<\cdots <
m^{(a_{l-1})}_{l-1}=\cdots =m^{(a^\diamondsuit)}_{l-1}$
for some $a_l<a_{l-1}<a^\diamondsuit$.
In particular, from $a_l\leq k+h_{l}$ and $m^{(a_l)}_{l-1}\geq 0$,
we have $m^{(k+h_l+1)}_{l-1}>0$.
(2) Recall that the first $h_l$ times $\delta_{l-1}$'s
remove totally $h_l$ length $l-1$
rows from each of $\nu^{(a)}$ ($k+h_l\leq a\leq a^\diamondsuit$).
Thus the length $l-1$ rows of $\nu^{(a)}$
($k+h_l+1\leq a\leq a^\diamondsuit$) remain singular after application
of $\delta_{l-1}^{h_l}$.
{}From the observations (1) and (2),
for the latter $h_{l-1}-h_l$ times $\delta_{l-1}$,
we can use the same arguments that we have used
in the case of $\delta_l^{h_l}$.
To summarize, we see that
the $i$-th $\delta_{l-1}$ removes length $l-1$ rows of
$\nu^{(a)},\ldots,\nu^{(a^\diamondsuit)}$ for some $a\leq k+i$.
Then by $a^\diamondsuit -k>\ell (\nu^{(a^\diamondsuit)})\geq h_{l-1}$
we have $k+h_{l-1}<a^\diamondsuit$, which implies that
$\delta_{l-1}^{h_{l-1}}$ will remove both $(l-1)$-th columns of
$\nu^{(a^\diamondsuit -1)}$ and $\nu^{(a^\diamondsuit)}$.
Thus $\delta_{l-1}^{h_{l-1}}$ remove the same positions of
$\nu^{(a^\diamondsuit -1)}$ and $\nu^{(a^\diamondsuit)}$.
Recursively, we conclude that every
$\delta_l,\delta_{l-1},\ldots,\delta_1$
removes from the same positions of
$\nu^{(a^\diamondsuit -1)}$ and $\nu^{(a^\diamondsuit)}$,
which concludes the proof of the proposition.
\end{proof}

\begin{lemma}\label{lem:shape_of_LR}
Consider the map 
$\Psi :(\nu^\bullet ,J^\bullet)\longmapsto
\{ (\nu'^\bullet ,J'^\bullet),T\}$.
Then $(\nu'^\bullet ,J'^\bullet)\in {\rm RC}^\varnothing$
and the outer shape of $T$ coincides with the weight of
$(\nu'^\bullet ,J'^\bullet)$.
\end{lemma}
\begin{proof}
Under the assumption (\ref{eq:condition_on_rank}) on the rank $n$,
we see from Proposition \ref{prop:Psi_stabilize} that
we can apply all $\delta$'s of $\Psi$ in well-defined manner.
Especially, from Proposition \ref{prop:delta_weldef}
we see that the image of each $\delta$
always belongs to ${\rm RC}^\diamondsuit$.
Recall that $\Psi$ will entirely remove $\nu^{(a)}$
for $a^\diamondsuit\leq a$.
Therefore the image $(\nu'^\bullet ,J'^\bullet)$
can be identified as the element of ${\rm RC}^\varnothing$.

Now we consider the statement about $T$.
Let us consider a particular step of $\Psi$,
say the application of $\delta_j$.
Then $\delta_j$ removes one box from each
$\nu^{(a)},\nu^{(a+1)},\ldots,\nu^{(n)}$ for some $a$.
On the level of the Young diagram that represents
the weight of the rigged configuration, this procedure
add a box to the $a$-th row of the Young diagram.
On the other hand, according to the definition of $\Psi$,
new $T$ is obtained by adding a box on the right of
the $a$-th row of $T$.
Since such coincidence occurs in every step of the $\Psi$,
we obtain the statement.
\end{proof}

\begin{remark}\label{rem:intermediate_LR}
{}From the above proof, we can also see that the outer shape of each
intermediate $T$ coincides with the weight of the corresponding
intermediate rigged configuration.
In particular shapes of all intermediate $T$ are Young diagrams.
Here we ignore the behavior of $\nu^{(a)}$ ($a>a^\diamondsuit$)
when $\diamondsuit =\vdom$.
\end{remark}

We prove the following technical lemma which is useful to
show that $T$ satisfies the definition of the
LR tableaux.

\begin{lemma}\label{lem:two_deltas}
Consider the application of $\delta_l^2$ on $(\nu^\bullet,J^\bullet)$
where $l=\nu^{(n)}_1$.
Then we have $l^{(a-1)}\leq l'^{(a)}$ where $a\leq a^\diamondsuit$.
Here we assume that $l^{(a)}=\infty$ (resp. $l'^{(a)}=\infty$)
if $\nu^{(a)}$ is not removed by the first (resp. second) $\delta_l$
and allow $\infty\leq\infty$.
\end{lemma}
\begin{proof}
As the statement about $\nu^{(a)}$
for $a^\diamondsuit <a$ is always true,
we assume $a\leq a^\diamondsuit$.
Thus the proof does not depend on $\diamondsuit$.
Also it is enough to show the statement under the strict
condition on the rank $n$ such as
$a^\diamondsuit>k+\ell (\nu^{(a^\diamondsuit)})$.
We proceed by induction on $a$.
By the condition on the rank $n$,
we can assume the shape of $(\nu^\bullet,J^\bullet)$
as described in Proposition \ref{prop:stability}.
{}From Proposition \ref{prop:Psi_stabilize},
we see that two $\delta_l$ remove the same positions of
$\nu^{(a^\diamondsuit -1)}$ and $\nu^{(a^\diamondsuit)}$.
In particular we have $l^{(a^\diamondsuit -1)}
=l'^{(a^\diamondsuit)}=l$, which proves the statement
for $a=a^\diamondsuit$.
Assume that we have $l^{(a)}\leq l'^{(a+1)}$
for some $a+1<a^\diamondsuit$.
Recall that by definition of $\delta_l$ we have
$l^{(a-1)}\geq l^{(a)}\geq l^{(a+1)}$.
Then the differences of the vacancy numbers by the first
$\delta_l$ is as follows:
$\Delta p^{(a)}_j=+1$ if $l^{(a-1)}>j\geq l^{(a)}$
and $\Delta p^{(a)}_j=0$ if $j\geq l^{(a-1)}$.
Note that if $\nu^{(a-1)}$ is not removed by the first $\delta_l$,
we can formally set $l^{(a-1)}=\infty$ and obtain the same relation.
In both cases, the length $j$ ($l^{(a-1)}>j\geq l^{(a)}$)
rows of $\nu^{(a)}$ are non-singular and cannot be
removed by the second $\delta_l$.
On the other hand, combining the assumed inequality $l'^{(a+1)}\geq l^{(a)}$
with $l'^{(a)}\geq l'^{(a+1)}$ that follows from
the definition of $\delta_l$, we have $l'^{(a)}\geq l^{(a)}$.
Therefore $l'^{(a)}$, i.e., the length of the row of $\nu^{(a)}$
that is removed by the second $\delta_l$
should satisfy $l'^{(a)}\geq l^{(a-1)}$.
Hence we complete the proof.
\end{proof}

\begin{cor}\label{cor:vertical_strip}
Let us consider intermediate steps of $\Psi$.
Let the recording tableau before (resp. after)
the application of $\delta_i^{h_i}$ be
$T$ (resp. $T'$).
Then the difference between $T'$ and $T$,
i.e., $T'\setminus T$,
forms the vertical strip of cardinality $h_i$
and, moreover, the letters contained in the strip
are $1,2,\ldots,h_i$ from top to bottom.
\end{cor}
\begin{proof}
Look at a pair of successive two $\delta_i$'s and apply
the previous Lemma \ref{lem:two_deltas} for
the corresponding intermediate rigged configuration.
Let us use symbols $l^{(a)}$ and $l'^{(a)}$ for
positions of boxes that are removed by $\delta_i$'s.
Suppose that $l^{(a)}<\infty$ and $l^{(a-1)}=\infty$.
Then we add a letter, say $j$, to the $a$-th row of
the recording tableau $T$.
{}From Lemma \ref{lem:two_deltas}, we have $l^{(a-1)}\leq l'^{(a)}$,
which forces $l'^{(a)}=\infty$.
In other words, the second $\delta_i$ cannot
remove from $\nu^{(a)}$.
Thus we have to put the letter $j+1$ to a row
that is strictly lower than the $a$-th row of $T$.
To summarize, $T'\setminus T$ contains at most
one element for each row, and if we read $T'\setminus T$
from top to bottom it reads $1,2,\ldots,h_i$.
Now recall from Lemma \ref{lem:shape_of_LR}
that shapes of all the intermediate recording tableaux
coincide with the corresponding
intermediate rigged configurations.
In particular their shapes are the Young diagrams.
Therefore $T'\setminus T$ forms a vertical strip
of cardinality $h_i$.
\end{proof}

\begin{lemma}\label{lem:comparison_LR}
Consider the applications of $\delta_{l}^{h_{l}}$
and $\delta_{l-1}^{h_{l-1}}$ of $\Psi$ for some $l$.
Let the positions (row, column) of the letter $c$ of $T$
associated with $\delta_{l}^{h_{l}}$
(resp. $\delta_{l-1}^{h_{l-1}}$) be $(i,j)$ (resp. $(i',j')$).
Then we have $i'\leq i$ and $j'>j$.
In other words, we have the following two possibilities;
$(i',j')$ is on the right of $(i,j)$, or
$(i',j')$ is strictly above and
strictly right of $(i,j)$.
\end{lemma}
\begin{proof}
We argue according to the computations of $\Psi$.
Thus $\delta_{l}^{h_{l}}$ appears first and
$\delta_{l-1}^{h_{l-1}}$ appears next.
Let us look at a particular pair $\delta_{l}^{2}$.
Denote the length of the row of $\nu^{(a)}$
that is removed by the first (resp. second)
$\delta_l$ by $l^{(a)}$ (resp. $l'^{(a)}$).
Let us tentatively call the row that
is removed by the first $\delta_l$ by $A$.
Then we observe the following two facts.
\begin{enumerate}
\item
If the row $A$ is removed by the first $\delta_l$,
then it will not be removed by all the remaining $\delta_l$.
In particular, the rigging of the row $A$ will not change
after removed by the first $\delta_l$.
To show this, let us look at the second $\delta_l$.
Recall that from Lemma \ref{lem:two_deltas}
we have $l^{(a)}\leq l'^{(a+1)}$ and from definition of
$\delta_l$ we have $l'^{(a)}\geq l'^{(a+1)}$,
thus $l'^{(a)}\geq l^{(a)}$.
On the other hand the length of the row $A$ after the first $\delta_l$,
thus it is $l^{(a)}-1$ that is strictly shorter than $l'^{(a)}$.
Hence it will not be removed by the second $\delta_l$.
Recursively we can show the claim.
\item
After the first $\delta_l$, the vacancy number for
the row $A$ will not be changed by the succeeding $\delta_l$'s.
To show this,
note that the length of the row $A$ after the first $\delta_l$ is $l^{(a)}-1$.
By Lemma \ref{lem:two_deltas}
we have $l^{(a)}\leq l'^{(a+1)}$ and from definition of
$\delta_l$ we have $l'^{(a-1)}\geq l'^{(a)}\geq l'^{(a+1)}$
so that the second $\delta_l$ will not change the vacancy
number for the row $A$.
Recursively we can show the claim.
\end{enumerate}
To summarize, once a row is removed by some $\delta_l$,
then the row is kept as singular after applications of
all the remaining $\delta_l$'s.
For the sake of the later arguments, we observe that
from $l^{(a-1)}\leq l'^{(a)}$ and
$l^{(a-1)}\geq l^{(a)}\geq l^{(a+1)}$,
we have $l^{(a-1)},l^{(a)},l^{(a+1)}\leq l'^{(a)}$.

Let us consider $\delta_{l-1}^{h_{l-1}}$.
To begin with, let us compare the locations of two letters 1 of $T$
corresponding to the first $\delta_l$ and $\delta_{l-1}$.
According to the conclusion in the previous paragraph,
we see that the first $\delta_{l-1}$ will remove a box from the same
rows that are removed by the first $\delta_l$.
Furthermore the first $\delta_{l-1}$ may remove
from leftward partitions that are not removed by
the first $\delta_l$.
Since the ``left" of the rigged configurations corresponds to
the ``up" of the $T$,
we obtain the statement in this case.
Now we shall check that the box removing procedure of the
first $\delta_{l-1}$ will not change the vacancy numbers for
the rows that are removed by $\delta_l$'s except for the
first $\delta_l$.
For this, remind that the first $\delta_{l-1}$ removes a box from
the same rows that are removed by the first $\delta_l$.
Thus we can use the inequality at the end of the last paragraph
to infer the property.
In fact, we have
$\Delta Q^{(a-1)}_{l'^{(a)}-1}=\Delta Q^{(a)}_{l'^{(a)}-1}=
\Delta Q^{(a+1)}_{l'^{(a)}-1}=-1$,
hence $\Delta p^{(a)}_{l'^{(a)}-1}=0$.
Thus we can use the same arguments to compare the second
$\delta_{l-1}$ with the second $\delta_l$ to check the statement
for this case.
Recursively, we can check the statement for all $\delta_{l-1}$,
and again recursively we can conclude the proof of the lemma.
\end{proof}

\begin{prop}\label{prop:yamanouchi}
Consider the map $\Psi :(\nu^\bullet,J^\bullet)
\longmapsto\{(\nu'^\bullet,J'^\bullet),T\}$.
Then the row word of $T$ satisfies the Yamanouchi condition.
\end{prop}
\begin{proof}
Consider the subset of $T$ corresponding to the columns
$h_l,h_{l-1},\ldots,h_j$ and proceed by induction on $j$.
For the initial case $j=l$,
we see from Corollary \ref{cor:vertical_strip}
that the row word is $123\cdots h_l$
which is the Yamanouchi word.
Assume that we have done for $h_l,h_{l-1},\ldots,h_{j+1}$.
Let the row word corresponding to the subset of $T$
corresponding to $h_l,h_{l-1},\ldots,h_{j+1}$ be $w$.
By the induction assumption $w$ is a Yamanouchi word.
As we see in Corollary \ref{cor:vertical_strip}, the column $h_j$
will insert integers $1,2,\ldots,h_j$ to $w$ in this order.
Let us rephrase Lemma \ref{lem:comparison_LR}
in the row word language. 
Then we see that each letter $i$
coming from the column $h_j$ will be inserted to
the left of all $i$'s contained in $w$.
Therefore insertion of the integers coming from
the column $h_j$ does not violate the Yamanouchi condition.
\end{proof}

\begin{proof}[Proof of Theorem \ref{Psi_weldef}.]
By Lemma \ref{lem:shape_of_LR}, $(\nu'^\bullet,J'^\bullet)$
is the type $\diamondsuit =\emptyset$ rigged configuration
and the outer shape of $T$ coincides with the weight
of $(\nu'^\bullet,J'^\bullet)$.
Moreover, from definition of $\Psi$ we see that
the inner shape of $T$ coincides with the weight
of $(\nu^\bullet,J^\bullet)$ and from Corollary
\ref{cor:vertical_strip} we see that the weight of $T$
coincides with $\nu^{(a^\diamondsuit)}$.
Thus we have only to check that the resulting $T$
indeed satisfies the definition of the LR tableau.
Recall that the heights of columns of $\nu^{(a^\diamondsuit)}$
satisfy $h_j\geq h_{j+1}$.
Recall also that the map $\Psi$ proceeds from $h_l$ to $h_1$.

Let us check that $T$ is a semi-standard tableau.
Fix a particular entry of $T$, say $\alpha$.
Assume that $\alpha$ corresponds to the column $h_j$.
Denote the possible its neighbours as
in the following diagram.
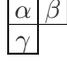
\[
\young(\alpha\beta,\gamma)
\]
If $\beta\neq\emptyset$, let us check $\alpha\leq\beta$.
Consider the next column $h_{j-1}$.
According to Lemma \ref{lem:comparison_LR}
there are two possibilities.
The first case is $\alpha=\beta$,
which immediately gives the statement.
The other case is that when computing the column $h_{j-1}$,
$\alpha$ appears strictly above and right
of the previous $\alpha$.
Then $\beta$ in the above diagram corresponds to
the column $h_i$ for some $i<j$.
By recursively using Lemma \ref{lem:comparison_LR}
we see that $\alpha$ corresponding to the column $h_i$
appears strictly above and right
of the $\alpha$ corresponding to the column $h_j$.
Recall from Corollary \ref{cor:vertical_strip} that
the entries corresponding to the column $h_i$ forms
a vertical strip and, moreover, its entry is strictly
increasing integer sequence if we read from top.
Within the vertical strip corresponding to the column
$h_i$, $\beta$ in the above diagram appears
strictly below $\alpha$.
Thus we have $\alpha <\beta$, which gives the statement
in this case.
Next, let us check $\alpha<\gamma$ if $\gamma\neq\emptyset$.
By Remark \ref{rem:intermediate_LR}, shape of each subset of
$T$ corresponding to $h_l,h_{l-1},\ldots,h_j$
is a Young diagram.
Thus $\gamma$ in the above diagram corresponds to
some column $h_i$ for some $i\leq j$.
If $\gamma$ corresponds to the same column $h_{j}$,
then we have $\gamma=\alpha+1>\alpha$,
which implies the statement in this case.
On the contrary, suppose that $\gamma$ corresponds to some
column $h_i$ ($i<j$).
In this case, again we can compare $\alpha$ and $\gamma$
within the vertical strip corresponding to the column $h_i$
to show $\alpha<\gamma$.
To summarize, we have checked that
$T$ is a semi-standard tableau.

Finally, from Proposition \ref{prop:yamanouchi}
we see that the row word of $T$ satisfies
the Yamanouchi condition.
Thus $T$ is the LR tableau.
\end{proof}

\subsection{Proof of Theorem \ref{Psi-1_weldef}}
During this section, 
we denote by $l^{(a)}$ the length of
the row of $\nu'^{(a)}$ to which $\tilde{\delta}_k$ has added a box.
In other words, $l^{(a)}$ is the position of the added box.
Then we have $l^{(a-1)}\geq l^{(a)}\geq l^{(a+1)}$.
Again, during the computations of $\tilde{\Psi}$
the quantum space does not change.
Therefore we can neglect the effect of the quantum space
when the arguments concern only with the differences of
the vacancy numbers.

\begin{prop}\label{prop:delta-1_weldef}
For $\diamondsuit =\hdom,\sdom$
$\tilde{\delta}_k$ is well-defined.
For $\diamondsuit =\vdom$
$\tilde{\delta}_k^2$ is well-defined.
\end{prop}
\begin{proof}
Again we have to check the positivity of the vacancy numbers
and the inequality that the riggings are less than or
equal to the corresponding vacancy numbers.
During the proof, we use $p_l'^{(a)}$ to express the vacancy
numbers with respect to the image of $\tilde{\delta}_k$ and
the symbol $\Delta$ to express the differences
of the quantities for after $\tilde{\delta}_k$
minus before $\tilde{\delta}_k$.\bigskip

\noindent
\underline{Case 1.}
In order to check the well-definedness for $\nu^{(a)}$
($a\geq a^\diamondsuit$), we can use case by case arguments
similar to those appeared in Proposition \ref{prop:delta_weldef}.
The rest of the proof does not depend on the specific choice of
$\diamondsuit$.
\bigskip

\noindent
\underline{Case 2.}
Next we check the well-definedness for rows of $\nu^{(a)}$
($a<a^\diamondsuit$)
that are not added by $\tilde{\delta}_k$.
Assume $l^{(a)}=\infty$ if $\nu^{(a)}$ is not
added by $\tilde{\delta}_k$.
Let $j$ be the length of a row of $\nu^{(a)}$.
Then $\Delta p^{(a)}_j$ behaves as follows:
$\Delta p^{(a)}_j=+1$ if $l^{(a)}>j\geq l^{(a+1)}$,
$\Delta p^{(a)}_j=-1$ if $l^{(a-1)}>j\geq l^{(a)}$
and $\Delta p^{(a)}_j=0$ otherwise.
Thus the only case that may cause difficulty is
$l^{(a-1)}>j\geq l^{(a)}$.
Let us call the length $j$ row by $A$
and let the corresponding rigging be $J(\geq 0)$.
By assumption the row $A$ is not added by $\tilde{\delta}_k$
so it was not singular before application of $\tilde{\delta}_k$,
i.e., $p^{(a)}_j-J>0$.
Thus even if we have $\Delta p^{(a)}_j=-1$,
we have $p'^{(a)}\geq J\geq 0$,
which assures the well-definedness.
\bigskip

\noindent
\underline{Case 3.}
Consider the remaining case, i.e.,
the rows of $\nu^{(a)}$ ($a<a^\diamondsuit$)
that are added by $\tilde{\delta}_k$.
In particular, we have to consider the case
$p^{(a)}_{l^{(a)}}=0$ before application of $\tilde{\delta}_k$.
If $l^{(a-1)}=l^{(a)}$, there is no $j$ such that
$\Delta p^{(a)}_j<0$ so that the well-definedness is assured.
Thus we assume that $l^{(a-1)}>l^{(a)}$ and lead to a contradiction.
If there is a length $l^{(a)}$ row of $\nu^{(a)}$,
then from $p^{(a)}_{l^{(a)}}=0$ and $l^{(a-1)}>l^{(a)}$,
$\tilde{\delta}_k$ would add a box to the length $l^{(a)}$ row.
This contradicts with the fact that $\tilde{\delta}_k$
add a box to a length $l^{(a)}-1$ row.
Thus there is not a length $l^{(a)}$ row of $\nu^{(a)}$.
Denote by $S$ the shortest row among the rows of $\nu^{(a)}$
that are strictly longer than $l^{(a)}-1$.
Let the length of the row $S$ be $s$.
If there is no such a row, assume $S=\emptyset$ and $s=\infty$.
Then from Lemma \ref{lem:convex} we see that
$p^{(a)}_{l^{(a)}-1}=p^{(a)}_{l^{(a)}}=\cdots
=p^{(a)}_{s}=0$ (see the similar arguments in
Proposition \ref{prop:delta_weldef}).
If $l^{(a-1)}\geq s$ then $\tilde{\delta}_k$ would add
a box to the row $S$ which is the contradiction.
Therefore we see that $l^{(a-1)}<s$.
Thus, for $l^{(a)}\leq i\leq s$, we have the following
properties as functions of $i$:
\begin{center}
\begin{tabular}{lcl}
$Q^{(a-1)}_i$&:&
strictly upper convex function due to the existence of
the length\\
&& $l^{(a-1)}$ row that satisfies
$l^{(a)}<l^{(a-1)}<s$.\\
$Q^{(a)}_i$&:& linear function,\\
$Q^{(a+1)}_i$&:&
upper convex function (including linear function case).
\end{tabular}
\end{center}
Combining these three contributions, we conclude
that $p^{(a)}_i$ is a strictly upper convex function
on the interval $l^{(a)}\leq i\leq s$.
However, as we have seen, $p^{(a)}_i=0$ holds
on this interval. This is the contradiction.
Hence we conclude $l^{(a-1)}=l^{(a)}$
which completes the proof of the proposition.
\end{proof}

\begin{lemma}\label{lem:twodelta-1s}
We follow the description at Definition \ref{def:Psi-1}.
According to Step (i) we fix a group of letters contained
in $T$ and consider a pair of successive two integers within the group.
Let the two letters of the pair be the $k$-th and the $k'$-th
rows of $T$ ($k>k'$), respectively.
Let us consider
$\tilde{\delta}_{k'}\circ\tilde{\delta}_{k}$.
Let $l^{(a)}$ (resp. $l'^{(a)}$)
be the column of the added box by $\tilde{\delta}_{k}$
(resp. $\tilde{\delta}_{k'}$).
As usual, assume $l^{(a)}=\infty$ (resp. $l'^{(k'-1)}=\infty$)
if $\nu^{(a)}$ is not added by $\tilde{\delta}_{k}$
(resp. $\tilde{\delta}_{k'}$).
Then, as long as $l'^{(a)}<\infty$,
we have $l'^{(a)}\leq l^{(a+1)}$.
\end{lemma}
\begin{proof}
We proceed by induction on $a$.
As the initial step, we show $l'^{(k-1)}\leq l^{(k)}$.
Recall that by $\tilde{\delta}_k$ the vacancy numbers
for the rows of $\nu^{(k-1)}$ that are longer than or equal
to $l^{(k)}$ are increased by 1, so that they are non singular
and cannot be added by $\tilde{\delta}_{k'}$.
This forces to $l'^{(k-1)}-1<l^{(k)}$, i.e., $l'^{(k-1)}\leq l^{(k)}$.
Note that $\tilde{\delta}_{k'}$ will add to length
$l'^{(k)}-1$ row of $\nu^{(k)}$.
As an induction hypothesis, for some $k\leq a$,
assume that we have $l'^{(a-1)}\leq l^{(a)}$.
Since the vacancy numbers for
the rows of $\nu^{(a)}$ whose lengths $j$
satisfy $l^{(a)}>j\geq l^{(a+1)}$ are increased by 1
due to $\tilde{\delta}_k$, they cannot be added by
$\tilde{\delta}_{k'}$.
Then, combining this with the inequality
$l^{(a)}>l'^{(a-1)}-1\geq l'^{(a)}-1$, we have
$l^{(a+1)}>l'^{(a)}-1$, i.e., $l^{(a+1)}\geq l'^{(a)}$.
By induction, we finish the proof of the lemma.
\end{proof}

\begin{lemma}\label{lem:delta-1_rig}
We follow the description at Definition \ref{def:Psi-1}.
According to Step (i) we fix a group of letters of $T$.
Once a row of $(\nu^\bullet,J^\bullet)$
becomes singular by the addition of
a box by $\tilde{\delta}$, then the row remains singular
during all the remaining procedure
corresponding to the same group of $T$.
\end{lemma}
\begin{proof}
As the proof is the same for all groups of letters of $T$,
we fix a group of the cardinality $h$ and call it the group $h$.
Let the letter under consideration be integer $j$ of the group $h$
at the $k$-th row of $T$.
Denote by $l^{(a)}$ (resp. $l'^{(a)}$) the column coordinate
of the box of $\nu^{(a)}$ that is added by $\tilde{\delta}_k$
(resp. $\tilde{\delta}_{k'}$), where $k'$ is the row coordinate
of the letter $j+1$ of the group $h$.
Let us check that the rigging corresponding to the row
that is added by $\tilde{\delta}_k$
will not be changed by the rest of the procedures
$\tilde{\delta}_{k'},\cdots$ corresponding to the rest of
the letters $j+1,j+2,\cdots,h$ of the group $h$.
{}From Lemma \ref{lem:twodelta-1s} we have $l'^{(a)}\leq l^{(a+1)}$
and from definition we have $l^{(a)}\geq l^{(a+1)}$
so that we have $l'^{(a)}-1<l^{(a)}$.
Thus $\tilde{\delta}_{k'}$ does not touch rows that are
touched by $\tilde{\delta}_{k}$, which implies
invariance of the rigging.
Recursively we see the invariance of the riggings
during all the remaining procedure
corresponding to the group $h$.

Next we check the invariance of the vacancy numbers.
Again, from Lemma \ref{lem:twodelta-1s} we have $l'^{(a-1)}\leq l^{(a)}$
and by definition we have $l'^{(a-1)}\geq l'^{(a)}\geq l'^{(a+1)}$.
Thus the vacancy number for the length $l^{(a)}$ rows of
$\nu^{(a)}$ will not change during the rest of the
procedures corresponding to the group $h$.
Combining the invariance of the riggings and the vacancy numbers,
the proof of the Lemma follows.
\end{proof}

\begin{proof}[Proof of Theorem \ref{Psi-1_weldef}]
Consider the map $\tilde{\Psi}:\{(\nu^\bullet,J^\bullet),T\}
\longmapsto (\nu'^\bullet,J'^\bullet)$.
In Proposition \ref{prop:delta-1_weldef} we checked
the well-definedness of $\tilde{\delta}_k$.
Thus the image $(\nu'^\bullet,J'^\bullet)$ is
a type $\diamondsuit$ rigged configuration.
Let us check that the weight of $(\nu'^\bullet,J'^\bullet)$
coincides with the inner shape of $T$.
By assumption the outer shape of $T$ coincides with the
weight of $(\nu^\bullet,J^\bullet)$.
Then observe the fact that each $\tilde{\delta}_k$
adds one box for each $\nu^{(a)}$ ($k\leq a$).
On the level of the weight, this means that
$\tilde{\delta}_k$ removes one box from the $k$-th
row of the Young diagram that represents the weight.
By construction of $\tilde{\Psi}$ we see that
application of $\tilde{\Psi}$ will remove from the Young
diagram all boxes which correspond to filled boxes of $T$.
Thus we see that the weight of $(\nu'^\bullet,J'^\bullet)$
is equal to the inner shape of $T$.

Finally we have to check that the shape of
weight of $T$ is equal to $\nu'^{(a^\diamondsuit)}$.
Let the cardinality of groups of $T$ be
$h_1,h_2,\cdots,h_l$ where the labellings are
according to the order in the procedure
of $\tilde{\Psi}$ (we will call the groups by their cardinality).
In particular, we have $h_i\geq h_{i+1}$.
Then what we have to show is that the columns of
$\nu'^{(a^\diamondsuit)}$ have height $h_1,h_2,\cdots,h_l$
from left to right.
The proof proceeds inductively on the number of groups
of $T$ that $\tilde{\Psi}$ has processed.
Consider the first group of the cardinality $h_1$.
Let the length of the outer shape of $T$ be $N$.
Then we have $\nu^{(a)}=\emptyset$ for all $N\leq a$.
Thus corresponding to the letter $h_1$,
$\tilde{\delta}$ will create a new row for
$\nu^{(a)}=\emptyset$ $(N\leq a)$.
By applying Lemma \ref{lem:twodelta-1s} with
$l^{(a)}=1$ $(N\leq a)$, we see that $l'^{(a)}=1$ $(N\leq a)$,
i.e., the second $\tilde{\delta}$ corresponding to
the letter $h_1-1$ will also create a new row for
$\nu^{(a)}$ $(N\leq a)$.
We can recursively continue the same argument and obtain
the following result: after processing all letters of
the group $h_1$ of $T$, $\nu^{(a)}$ $(N\leq a)$ becomes
the single column type partition whose height is $h_1$.
Thus we have checked the assertion for the first group $h_1$.

Suppose that we have shown the property for the group $h_i$.
As the induction hypothesis, each $\nu^{(a)}$ $(N\leq a)$
has columns whose heights are $h_1,h_2,\cdots,h_i$
from left to right.
Let us apply Lemma \ref{lem:delta-1_rig} in this situation.
Then all rows that are added by $\tilde{\delta}$
corresponding to all letters of the group $h_i$
remain as singular rows.
Recall that, by definition of $\tilde{\Psi}$,
the letter $j$ of the group
$h_{i+1}$ is located on the same row or lower row of $T$
compared with the letter $j$ of the group $h_i$.
To begin with let us consider the letter $h_{i+1}$
of the group $h_{i+1}$.
Recall that $h_i\geq h_{i+1}$ so that there is also
the letter $h_{i+1}$ within the group $h_i$.
Then, on the level of $(\nu^\bullet,J^\bullet)$,
the process corresponding with the letter $h_{i+1}$
of the group $h_{i+1}$ start from the same $\nu^{(a)}$
or rightward $\nu^{(a)}$ compared with the starting point
of the process corresponding to the letter $h_{i+1}$
of the group $h_{i}$.
Combining this with Lemma \ref{lem:delta-1_rig},
we conclude that the operation $\tilde{\delta}$ corresponding
to the letter $h_{i+1}$ of the group $h_{i+1}$ will
add one more box to the same rows that are added by
the process related with the letter $h_{i+1}$ of the group $h_{i}$.
Especially we add the new column, i.e., the $(i+1)$-th column
to each $\nu^{(a)}$ $(N\leq a)$.
Observe that the procedure corresponding to
the letter $h_{i+1}$ of the group $h_{i+1}$ will not violate
the singular property of all rows that are added and set to be
singular by the letters $h_{i+1}-1,\cdots,2,1$ of the group $h_i$.
This follows from the fact that the procedure for the letter $h_{i+1}$
of the group $h_{i+1}$ exactly follow that for the group $h_i$,
hence we can apply Lemma \ref{lem:twodelta-1s} to claim the
invariance of the vacancy numbers.
Thus we can continue recursively the same arguments and
obtain the following result: after processing all letters
of the group $h_{i+1}$, each $\nu^{(a)}$ $(N\leq a)$ has
$i+1$ columns whose heights are $h_1,h_2,\cdots,h_{i+1}$
from left to right.
This completes the induction step thereby finishing
the whole proof of Theorem \ref{Psi-1_weldef}.
\end{proof}

\section{$M^\dia$ in terms of $M^\varnothing$}
In this section we assume $\dia\ne\varnothing$. Note that $\gamma$ defined in section 
\ref{subsec:stability} is given by $\gamma=2/|\dia|$ for $\geh^\dia$.

\subsection{Change of statistic}
Let $(\nu^\bullet,J^\bullet)\in\mathrm{RC}^\dia$. Suppose the map $\delta$ sends 
$(\nu^\bullet,J^\bullet)$ to $((\tilde{\nu}^\bullet,\tilde{J}^\bullet),k)$. Recall $l^*$ 
in Proposition \ref{prop:stability}. Let $\eta_a$ be the length of the row
of $\nu^{(a)}$ whose rightmost node is removed by $\delta$ for $k\le a\le l^*$.

\begin{lemma} \label{lem:c diff}
\begin{align*}
(c(\nu)-c(\tilde{\nu}))/\gamma=
& 2\sum_{k\le a<l^* \atop i\ge\eta_a}m_i^{(a)}-\sum_{k-1\le a<l^* \atop i\ge\eta_{a+1}}m_i^{(a)}
-\sum_{k\le a<l^* \atop i\ge\eta_a}m_i^{(a+1)}+\sum_{i\ge\eta_{l^*}}m_i^{(l)} \\
& -\sum_{k\le a<l^*}(1-\delta_{\eta_a\eta_{a+1}})-\frac12-\sum_{k\le a\le l^* \atop i\ge\eta_a}L_i^{(a)}
\end{align*}
\end{lemma}

\begin{proof}
Note that $m_i^{(a)}$ changes to $m_i^{(a)}-\delta_{i\eta_a}+\delta_{i,\eta_a-1}$ by $\delta$ and
use Proposition \ref{prop:stability} (3). A direct calculation shows the desired result.
\end{proof}

Let $\tilde{p}_i^{(a)}$ be the vacancy number of $(\tilde{\nu}^\bullet,\tilde{J}^\bullet)$.
\begin{lemma} \label{lem:vac diff}
For $k\le a\le l^*$,
\[
p^{(a)}_{\eta_a}-\tilde{p}^{(a)}_{\eta_a-1}=\sum_{i\ge\eta_a}L^{(a)}_i
+\sum_{i\ge\eta_a}m^{(a-1)}_i-2\sum_{i\ge\eta_a}m^{(a)}_i 
+\sum_{i\ge\eta_a}m^{(a+1)}_i+1-\delta_{\eta_a\eta_{a+1}}.
\]
\end{lemma}

\begin{proof}
Direct calculation noting that $\eta_{a-1}\ge\eta_a\ge\eta_{a+1}$.
\end{proof}

\begin{prop} \label{prop:ch diff}
Suppose we get $(\tilde{\nu}^\bullet,\tilde{J}^\bullet)$ from $(\nu^\bullet,J^\bullet)$ by the map
$\delta$. Then we have $c(\nu^\bullet,J^\bullet)-c(\tilde{\nu}^\bullet,\tilde{J}^\bullet)=-\gamma/2$.
\end{prop}

\begin{proof}
Let $\tilde{p}_i^{(a)}$ be the vacancy number of $(\tilde{\nu}^\bullet,\tilde{J}^\bullet)$. By the 
definition of $\delta$, we have 
$c(\nu^\bullet,J^\bullet)-c(\tilde{\nu}^\bullet,\tilde{J}^\bullet)
=(c(\nu)-c(\tilde{\nu}))+\gamma\sum_{k\le a\le l^*}(p^{(a)}_{\eta_a}-\tilde{p}^{(a)}_{\eta_a-1})$
where $\eta_a$ is the length of the row of $\nu^{(a)}$ whose rightmost node is removed by $\delta$ 
for $k\le a\le l^*$. A direct calculation using Lemma \ref{lem:c diff} and \ref{lem:vac diff} completes
the proof.
\end{proof}

Now we have the following theorem.

\begin{theorem} \label{th:M=K}
For $\dia=\cell,\hdom,\vdom$ the stable fermionic formula $M^\dia(\la,\El;q)$ is expressed as a sum
of $M^\varnothing(\eta,\El;q)$ as follows.
\[
M^\dia(\la,\El;q)=q^{-\frac{\gamma}2(|\El|-|\la|)}
\sum_{\mu\in\mathcal{P}^\dia_{|\El|-|\la|},\eta\in\mathcal{P}^\cells_{|\El|}}c_{\la\mu}^\eta
M^\varnothing(\eta,\El;q^\gamma)
\]
Here $|\El|=\sum_{a\in I_0,i\in\Z_{>0}}aiL_i^{(a)},\mathcal{P}^\dia_N$ is the set of
partitions of $N$ whose diagrams can be tiled by $\dia$, and $c_{\la\mu}^\eta$ is the 
Littlewood-Richardson coefficient.
\end{theorem}

\begin{proof}
In view of Theorem \ref{th:main}, Proposition \ref{prop:ch diff} and the fact that the image of $\Psi$ can
be regarded as a rigged configuration of type $\varnothing$ (although the statistic is multiplied 
by $\gamma$), it is sufficient to show
\begin{itemize}
\item[(i)] $\delta$ increases $|\la|$ by 1,
\item[(ii)] in the image of $\Psi$ we have $|\la|=|\El|$
\end{itemize}
to prove the theorem.

Let $\delta$ remove a box from $\nu^{(a)}$ for $a\ge k$. Then, from \eqref{config} it amounts to
adding $\alpha_k+\alpha_{k+1}+\cdots+\alpha_n/\upsilon_n=\epsilon_k\,(=\,$standard orthonormal basis
of the weight space) to $\la$ when $\dia=\cell,\hdom$. Hence, it increases $|\la|$ by 1. If 
$\dia=\vdom$, we see applying $\delta$ two times increases $|\la|$ by 2.

Next we show (ii). In the image of $\Psi$, when $k\ge\ell(\wt(\nu^\bullet,J^\bullet))$, 
from \eqref{num of box} we have 
\[
0=\frac2{(\alpha_k|\alpha_k)}\sum_{b,i}(iL_i^{(b)}-\la_b)(\Lab_k|\Lab_b).
\]
Since$(\Lab_k|\Lab_b)>0$ for any $b\in I_0$, we obtain $\la_b=\sum_iiL_i^{(b)}$ for any $b\in I_0$, 
which concludes $|\la|=\sum_bb\la_b=|\El|$.
\end{proof}

\begin{remark}
The minimum rank $n$ that makes the theorem hold is determined by the condition that 
\eqref{eq:condition_on_rank} is satisfied for any $\mu=\nu^{(a^\dia)}$ such that $c_{\la\mu}^\eta>0$.
It is in fact given by
\[
\ell(\la)+\frac{|\El|-|\la|}{\mathrm{width}(\dia)}\le a^\dia.
\]
\end{remark}


\begin{thebibliography}{99}

\bibitem{DK}
P.~Di~Francesco and R.~Kedem,
\textit{Proof of the combinatorial Kirillov-Reshetikhin conjecture},
Int. Math. Res. Notices, (2008) Volume 2008: article ID rnn006, 57 pages.

\bibitem{FOS}
G.~Fourier, M.~Okado and A.~Schilling,
\textit{Kirillov-Reshetikhin crystals for nonexceptional types},
Adv. in Math. {\bf 222} (2009), 1080--1116.

\bibitem{FOY} K.~Fukuda, M.~Okado and Y.~Yamada,
Energy functions in box-ball systems,
Int. J. Mod. Phys. {\bf A15} (2000) 1379--1392.

\bibitem{F}
W.~Fulton, 
\textit{Young tableaux}, London Mathematical Society Student Texts \textbf{35}, 
Cambridge University Press (1997).

\bibitem{HHIKTT} G.~Hatayama, K.~Hikami, R.~Inoue, A.~Kuniba,
T.~Takagi and T.~Tokihiro,
The $A^{(1)}_M$ automata related to crystals of symmetric tensors,
J. Math. Phys. {\bf 42} (2001) 274--308.

\bibitem{HKOTT}
G.~Hatayama, A.~Kuniba, M.~Okado, T.~Takagi and Z.~Tsuboi,
\textit{Paths, crystals and fermionic formulae}, MathPhys Odyssey 2001, 205--272, 
Prog.\ Math.\ Phys.\ {\bf 23}, Birkh\"auser Boston, Boston, MA, 2002.

\bibitem{HKOTY}
G.~Hatayama, A.~Kuniba, M.~Okado, T.~Takagi and Y.~Yamada, 
\textit{Remarks on fermionic formula},
Contemporary Math.\ {\bf 248} (1999) 243--291.

\bibitem{HKT} G.~Hatayama, A.~Kuniba and T.~Takagi,
Soliton cellular automata associated with crystal bases,
Nucl. Phys. {\bf B577} (2000) 619--645.

\bibitem{Her}
D.~Hernandez,
\textit{Kirillov-Reshetikhin conjecture: the general case},
Int. Math. Res. Notices (2010) no. 1, 149--193. 

\bibitem{Kac}
V.~G.~ Kac, 
\textit{``Infinite Dimensional Lie Algebras,"}
3rd ed., Cambridge Univ. Press, Cambridge, UK, 1990.

\bibitem{KKR}
S.~V.~Kerov, A.~N.~Kirillov and N.~Yu.~Reshetikhin,
\textit{Combinatorics, the Bethe ansatz and representations of the symmetric group},
Zap.Nauchn. Sem. (LOMI) {\bf 155} (1986) 50--64
(English translation: J. Sov. Math. {\bf 41} (1988) 916--924).

\bibitem{KR}
A.~N.~Kirillov and N.~Yu.~Reshetikhin,
\textit{The Bethe ansatz and the combinatorics of Young tableaux},
J. Sov. Math. {\bf 41} (1988) 925--955.

\bibitem{KSS}
A.~N.~Kirillov, A.~Schilling and M.~Shimozono,
\textit{A bijection between Littlewood-Richardson tableaux and rigged configurations}, 
Selecta Math. (N.S.) {\bf 8} (2002) 67--135.

\bibitem{KNT}
A.~Kuniba, T.~Nakanishi and Z.~Tsuboi,
\textit{The canonical solutions of the $Q$-systems and the Kirillov-Reshetikhin conjecture},
Comm. Math. Phys. \textbf{227} (2002) 155--190.

\bibitem{KOSTY} A.~Kuniba, M.~Okado, R.~Sakamoto,
T.~Takagi and Y.~Yamada,
Crystal interpretation of Kerov--Kirillov--Reshetikhin bijection,
Nucl. Phys. {\bf B740} (2006) 299--327.

\bibitem{KSY1} A.~Kuniba, R.~Sakamoto and Y.~Yamada,
Tau functions in combinatorial Bethe ansatz,
Nucl. Phys. {\bf B786} (2007) 207--266.

\bibitem{KSY2}  A.~Kuniba, R.~Sakamoto and Y.~Yamada,
Generalized energies and integrable $D^{(1)}_n$ cellular automaton,
arXiv:1001.1813 (to be published from World Scientific Inc.).

\bibitem{LS}
C.~Lecouvey and M.~Shimozono,
\textit{Lusztig's $q$-analogue of weight multiplicity and one-dimensional sums 
for affine root systems},
Adv. in Math. \textbf{208} (2007) 438--466.

\bibitem{LOS}
C.~Lecouvey, M.~Okado and M.~Shimozono,
\textit{Affine crystals, one-dimensional sums and parabolic Lusztig $q$-analogues},
arXiv:1002.3715.

\bibitem{Nak}
H.~Nakajima,
\textit{$t$-analogues of $q$-characters of Kirillov-Reshetikhin modules of 
quantum affine algebras},
Represent. Theory \textbf{7} (2003) 259--274.


\bibitem{O}
M.~Okado,
\textit{Existence of crystal bases for Kirillov-Reshetikhin modules of type $D$},
Publ. RIMS \textbf{43} (2007) 977-1004.

\bibitem{OS}
M.~Okado and A.~Schilling,
\textit{Existence of Kirillov-Reshetikhin crystals for nonexceptional types},
Representation Theory \textbf{12} (2008) 186--207.

\bibitem{OSS}
M.~Okado, A.~Schilling and M.~Shimozono,
{\it A crystal to rigged configuration bijection for nonexceptional affine algebras}, in 
{\it Algebraic Combinatorics and Quantum Groups},
eds. N. Jing, (World Scientific 2003), pp85--124. 

\bibitem{Sak1} R.~Sakamoto, 
Crystal interpretation of Kerov--Kirillov--Reshetikhin bijection II.
Proof for $\mathfrak{sl}_n$ case,
J. Algebraic Combin. {\bf 27} (2008) 55--98.

\bibitem{Sch}
A.~Schilling,
\textit{A bijection between type $D_n^{(1)}$ crystals and rigged configurations},
J. Algebra {\bf 285} (2005) 292--334.

\bibitem{SS}
A.~Schilling and M.~Shimozono,
{\it $X=M$ for symmetric powers},
J. Algebra {\bf 295} (2006) 562--610.

\bibitem{Sh}
M.~Shimozono,
\textit{On the $X=M=K$ conjecture},
arXiv:math.CO/0501353.

\bibitem{SZ}
M.~Shimozono and M.~Zabrocki,
\textit{Deformed universal characters for classical and affine algebras},
J. of Algebra \textbf{299} (2006) 33--61.

\end{thebibliography}
\end{document}